\documentclass[sn-mathphys,a4paper,Numbered]{sn-jnl}

\usepackage{amsmath,amssymb,amsfonts}%
\usepackage{amsthm}%
\usepackage{bbold}
\usepackage{color}
\usepackage{comment}
\usepackage{hyperref}
\usepackage{graphicx,epsfig}
\usepackage{mathrsfs}%
\usepackage{manyfoot}%
\usepackage{listings}%

\definecolor{darkred}{RGB}{139,0,0}
\definecolor{darkgreen}{RGB}{0,100,0}
\definecolor{darkmagenta}{RGB}{139,0,139}
\definecolor{darkpurple}{RGB}{110,0,180}
\definecolor{darkblue}{RGB}{40,0,200}
\definecolor{darkorange}{RGB}{255,140,0}

\newcommand{\bsell}{\boldsymbol{\ell}}
\newcommand{\bsa}{\boldsymbol{a}}
\newcommand{\bsb}{\boldsymbol{b}}
\newcommand{\bsc}{\boldsymbol{c}}

\newcommand{\bsgamma}{\boldsymbol{\gamma}}
\newcommand{\bsGamma}{\boldsymbol{\Gamma}}
\newcommand{\bsx}{\boldsymbol{x}}
\newcommand{\bsz}{\boldsymbol{z}}

\newcommand{\bsw}{\boldsymbol{w}}

\newcommand{\bsm}{\boldsymbol{m}}

\newcommand{\bsk}{\boldsymbol{k}}

\newcommand{\bsy}{\boldsymbol{y}}

\newcommand{\bsN}{\boldsymbol{N}}
\newcommand{\rd}{\,{\rm d}}

\newcommand{\RR}{\mathbb{R}}
\newcommand{\NN}{\mathbb{N}}

\newcommand{\FF}{\mathbb{F}}
\newcommand{\cP}{\mathcal{P}}

\newcommand{\cL}{\mathcal{L}}
\newcommand{\cN}{\mathcal{N}}
\newcommand{\cK}{\mathcal{K}}
\newcommand{\bscK}{\boldsymbol{\mathcal{K}}}

\newcommand{\uu}{\mathfrak{u}}

\newcommand{\bI}{\mathbf{I}}
\newcommand{\bU}{\mathbf{U}}
\newcommand{\bmu}{\boldsymbol{\mu}}
\newcommand{\bSigma}{\boldsymbol{\Sigma}}
\newcommand{\blambda}{\boldsymbol{\lambda}}
\newcommand{\bLambda}{\boldsymbol{\Lambda}}

\theoremstyle{thmstyleone}%
\newtheorem{theorem}{Theorem}

\newtheorem{proposition}{Proposition}
\newtheorem{lemma}{Lemma}

\theoremstyle{thmstyletwo}%
\newtheorem{remark}{Remark}

\theoremstyle{thmstylethree}%
\newtheorem{definition}{Definition}

\allowdisplaybreaks

\begin{document}

\title[running title]{Quasi-Monte Carlo methods for mixture distributions and approximated distributions via piecewise linear interpolation}

\author[1]{\fnm{Tiangang} \sur{Cui}}\email{tiangang.cui@monash.edu}

\author[2]{\fnm{Josef} \sur{Dick}}\email{josef.dick@unsw.edu.au}

\author[3]{\fnm{Friedrich} \sur{Pillichshammer}}\email{friedrich.pillichshammer@jku.at}

\affil[1]{\orgdiv{School of Mathematics}, \orgname{Monash University}, \orgaddress{\street{9 Rainforest Walk}, \city{Melbourne}, \postcode{VIC 3800}, \country{Australia}}}

\affil[2]{\orgdiv{School of Mathematics and Statistics}, \orgname{The University of New South Wales}, \orgaddress{\city{Sydney}, \postcode{NSW 2052}, \country{Australia}}}

\affil[3]{\orgdiv{Institut f\"{u}r Finanzmathematik und Angewandte Zahlentheorie}, \orgname{Johannes Kepler Universit\"{a}t Linz}, \orgaddress{\street{Altenbergerstra{\ss}e 69}, \city{Linz}, \postcode{A-4040}, \country{Austria}}}

\abstract{
We study numerical integration over bounded regions in $\mathbb{R}^s$, $s \ge 1$, with respect to some probability measure. We replace random sampling with quasi-Monte Carlo methods, where the underlying point set is derived from deterministic constructions which aim to fill the space more evenly than random points. Ordinarily, such quasi-Monte Carlo point sets are designed for the uniform measure, and the theory only works for product measures when a coordinate-wise transformation is applied. Going beyond this setting, we first consider the case where the target density is a mixture distribution where each term in the mixture comes from a product distribution. Next we consider target densities which can be approximated with such mixture distributions. In order to be able to use an approximation of the target density, we require the approximation to be a sum of coordinate-wise products and that the approximation is positive everywhere (so that they can be re-scaled to probability density functions). We use tensor product hat function approximations for this purpose here, since a hat function approximation of a positive function is itself positive. 

We also study more complex algorithms, where we first approximate the target density with a general Gaussian mixture distribution and approximate the mixtures with an adaptive hat function approximation on rotated intervals. The Gaussian mixture approximation allows us (at least to some degree) to locate the essential parts of the target density, whereas the adaptive hat function approximation allows us to approximate the finer structure of the target density.

We prove convergence rates for each of the integration techniques based on quasi-Monte Carlo sampling for integrands with bounded partial mixed derivatives. The employed algorithms are based on digital $(t,s)$-sequences over the finite field $\FF_2$ and an inversion method. Numerical examples illustrate the performance of the algorithms for some target densities and integrands.
} 

\keywords{Weighted integration, quasi-Monte Carlo, digital sequences, mixture distributions}

\maketitle

\section{Introduction}

We are interested in the numerical approximation of $\overline{\pi}$-weighted integrals \begin{equation}\label{wintunit}
\int_D f(\bsx) \overline{\pi}(\bsx) \rd \bsx
\end{equation}
of functions $f$ over certain bounded domains $D \subset \mathbb{R}^s$ (including rectangular domains, convex domains and others), where $\overline{\pi}$ is a probability density function (PDF).  

Many statistical methods have been developed to approximate such integrals. Depending on what is known about $\overline{\pi}$, one can for instance use a Monte Carlo method (in case one can sample directly from $\pi$), or a Markov chain Monte Carlo method \cite{Liu2008,RG2004} (for instance, when the $\overline{\pi}$ is only known up to a normalizing constant), or approximate Bayesian computation \cite{sisson2007sequential} (when the underlying likelihood function is hard or impossible to compute). Further possibilities are support points \cite{MJ18}, Stein point methods \cite{pmlr-v80-chen18f}, minimum energy points \cite{JMY90}, transport maps \cite{MMPS16}, as well as other statistical methods \cite{HLD04}.

If the integrand is smooth and the density is a product of one dimensional densities whose cumulative distribution function can be inverted, then one method for approximating integrals of the form \eqref{wintunit} is using a quasi-Monte Carlo (QMC) method. These are equal weight quadrature rules to approximate integrals over the unit cube $[0,1]^s$ of the form
\begin{equation}\label{QMC_approx}
\int_{[0,1]^s} f(\bsx) \rd \bsx \approx \frac{1}{N} \sum_{n=0}^{N-1} f(\bsx_n),
\end{equation}
where $\bsx_0, \bsx_1, \ldots, \bsx_{N-1} \in [0,1]^s$ are deterministically chosen quadrature points. See \cite{DKP22, DKS13, DP10, L09, LP14, niesiam} for detailed general information about QMC methods.

Most of the theoretical results for QMC apply to integrands defined on $[0,1]^s$. If one wants to apply the theory to other domains $D$, a transformation from $D$ to $[0,1]^s$ has to be used to transform the problem to a problem over the unit cube. One case where this can be done successfully is for integrands defined on $\RR^s$ where the integrand is with respect to, for instance, the normal distribution; see \cite{KDSWW08, NK14} (consult the references there for examples of other distributions). The transformation plays a critical role in obtaining good convergence rates, but this also limits the applicability of these methods to a few (albeit important) examples. Other examples are integrals over a triangle \cite{BO15, GSY17}, or over the sphere \cite{ABD12}. A combination of statistical sampling where the random driver sequence is replaced by randomized QMC samples was studied in \cite{GC15, GC17}. Another approach in this direction is the array RQMC method \cite{LMT10}. One of the major differences between \cite{GC15} and \cite{LMT10} is the use of a space filling curve, which was further studied in \cite{SHGCN16}.

In \cite{DAFS19}, QMC was used in conjunction with importance sampling. The numerical results therein comparing various methods indicated that QMC methods work quite well and that this may be an approach warranting further consideration.

In the paper \cite{DP2020} the case where $\overline{\pi}$ is of product form, i.e., a product of univariate PDFs, and where \eqref{wintunit} is approximated by a suitably transformed quasi-Monte Carlo rule was studied. In the present paper we extend the approach of \cite{DP2020} to mixture distributions and distributions that can be sufficiently well approximated by mixture distributions.

In the first instance, in Section~\ref{sec:QMCmix} we consider mixture distributions which are a sum of product distributions. For each mixture, we can use a coordinate-wise transformation based on the inverse cumulative distribution function to transform QMC points, where the number of QMC points for each mixture depends on the weight given to that particular mixture. Using the main result from \cite{DP2020} we obtain a bound for the integration error in this case. 

In Section~\ref{sec4:QMChat} we use a hat function approximation of the target density. This approximation can be interpreted as a mixture distribution so that the results from Section~\ref{sec:QMCmix} still apply. An additional feature here is that we do not need to know the normalizing constant of the probability density $\overline{\pi}$, i.e., we assume we know only $\pi$ such that $\overline{\pi} = \pi / z$, for some unknown positive real number $z$. Although we do not know the normalizing constant, we do know how to normalize the hat functions to a probability function and we can numerically calculate the weight for each of the (one-dimensional products) of the hat functions to turn them into probability densities. The coefficient of the hat function normalized by the sum of all the coefficients is then used to determine how many QMC points we transform for this particular hat function probability density function.

Finally, in Section~\ref{sec5:partition}, we also introduce an adaptive hat function approximation, which allows one to use an adaptive refinement of the hat function approximation. The hat functions we consider are products of one-dimensional hat functions. Their support is an interval. In order to be able to better capture the structure of the target density, we also consider hat functions defined on rotated intervals. To do so, we use a linear transformation from the unit cube to the support of such hat functions. We use this more general form of a hat function approximation to the target density to better capture the important parts of the target density. Additionally, we first use a partition of unity approximation of the target density and apply the general hat function approximation to each of the parts of the partition of unity approximation of the target density. Again, we assume that the target density is only known up to the normalizing constant. 

Numerical experiments in Section~\ref{sec:numerical} illustrate the theoretical results.

\section{Background on QMC methods}

We provide some background on the QMC theory used in this paper.

\subsection{Function space setting}

QMC rules allow one to obtain convergence rates of the worst-case integration error of order $N^{-1} (\log N)^{s-1}$, where $N$ is the number of quadrature points and $s$ is the dimension, if the integrand satisfies some smoothness assumptions (see \cite{DKP22, DKS13, DP10, L09, LP14, niesiam}). In the following we state the smoothness assumptions on the integrand which we require in this paper.

We consider functions over an interval $$[\bsa,\bsb] =[a_1,b_1]\times \ldots \times [a_s,b_s],$$ where $\bsa=(a_1,\ldots,a_s)$ and $\bsb=(b_1,\ldots,b_s)$ are in $\mathbb{R}^s$, with bounded mixed partial derivatives up to order one in the $\sup$-norm and define, for functions $f$ defined on $[\bsa,\bsb]$ and for $p \in [1,\infty]$, the ``$\bsgamma$-weighted'' $p$-norm
\begin{equation}\label{def:norm}
\| f \|_{p, s,\bsgamma}:= \left( \sum_{\uu \subseteq [s]} \left( \frac{1}{\gamma_\uu} \sup_{\bsx \in [\bsa, \bsb]} \left | \frac{\partial^{|\uu|} f}{\partial \bsx_\uu} (\bsx)  \right | \right)^p \right)^{1/p},
\end{equation}
where $\bsgamma=\{\gamma_{\uu}\}_{\uu \subseteq [s]}$ (throughout the paper we abbreviate $[s]:=\{1,2,\ldots,s\}$) is a set of positive real weights, with the obvious modifications if $p = \infty$. Here, for $\uu=\{u_1,u_2,\ldots,u_k\}$ with $1 \le u_1 < u_2 < \ldots < u_k \le s$, we write $$\frac{\partial^{|\uu|} f}{\partial \bsx_\uu}:= \frac{\partial^k f}{\partial x_{u_1} \partial x_{u_2} \ldots \partial x_{u_k}}.$$ Note that for every $p \in [1,\infty]$ we have $\|f\|_{L_\infty} \le \gamma_{\emptyset} \| f \|_{p, s,\bsgamma}$ (typically, e.g., for product weights, $\gamma_{\emptyset}$ equals 1).

As usual, $\mathbb{N}$ denotes the set of positive integers and $\mathbb{N}_0:=\mathbb{N} \cup \{0\}$.

\subsection{Digital nets and sequences}

In this section we describe the construction method for the points used in the QMC rule given by \eqref{QMC_approx}.

Let $\FF_2$ be the finite field of order $2$. We identify $\FF_2$ with the set $\{0,1\}$ equipped with arithmetic operations modulo 2.

We introduce a particular type of QMC point set called digital net over $\FF_2$, and its infinite extension called digital sequence over $\FF_2$. These definitions were given first by Niederreiter~\cite{nie87}.

\begin{definition}\rm \label{def:dignet}
Let $s \ge 1$, $m \ge 1$ and $0 \le t \le m$ be integers. Choose
 $m \times m$ matrices $C_1,\ldots ,C_s$ over $\FF_2$ with the following property:

For any non-negative integers $d_1,\ldots ,d_s$ with $d_1+\cdots+d_s=m-t$ the system of the
\begin{center}
\begin{tabbing}
\hspace*{3cm}\=first \= $d_1\ \ \ $ \= rows of $C_1$, \=  together with the\\
\> first \> $d_2$ \> rows of $C_2$, \=  together with the \\
\>\hspace{1cm}$\vdots$ \\
\> first \> $d_{s-1}$ \> rows of $C_{s-1}$, \=  together with the \\
\>first \> $d_s$ \>rows of $C_s$
\end{tabbing}
\end{center}
is linearly independent over $\FF_2$.

Consider the following construction principle for point sets consisting of $2^m$ points in $[0,1)^s$: represent $n \in \{0,1,\ldots,2^m-1\}$ in base 2, say $n=n_0+n_1 2+\cdots +n_{m-1} 2^{m-1}$ with binary digits $n_0,n_1,\ldots,n_{m-1} \in \{0,1\}$, and multiply for every $j \in [s]$ the matrix $C_j$ with the vector $\vec{n} = (n_0,\ldots,n_{m-1})^{\top}$ of digits of $n$ in $\FF_2$,
\begin{eqnarray}\label{matrix_vec_net}
C_j \vec{n}=:(y_1^{(j)},\ldots ,y_m^{(j)})^{\top}.
\end{eqnarray}
Now we set
\begin{eqnarray}\label{defxni}
x_n^{(j)}:=\frac{y_1^{(j)}}{2}+ \cdots +\frac{y_m^{(j)}}{2^m}
\end{eqnarray}
and
\begin{eqnarray}
\bsx_n = (x_n^{(1)}, \ldots ,x_n^{(s)}).\nonumber
\end{eqnarray}
The point set $\{\bsx_0,\bsx_1,\ldots,\bsx_{2^m-1}\}$ is called a {\it digital $(t,m,s)$-net over $\FF_2$} and the matrices $C_1,\ldots,C_s$ are called the {\it generating matrices} of the digital net.
\end{definition}

We remark that explicit constructions for digital $(t,m,s)$-nets are known with some restrictions on the parameter $t$ (the so-called quality parameter $t$ can be independent of $m$ but depends on $s$), see for instance \cite{DP10,LP14,niesiam} for more information.\\

Digital sequences are infinite versions of digital nets.

\begin{definition}\rm
Let $C_1,\ldots, C_s \in \mathbb{F}_2^{\mathbb{N} \times \mathbb{N}}$ be $\mathbb{N} \times \mathbb{N}$ matrices over $\mathbb{F}_2$. For $j \in [s]$ and $C_j = (c_{j,k,\ell})_{k, \ell \in \mathbb{N}}$ we assume that for each $\ell \in \mathbb{N}$ there exists a $K(\ell) \in \mathbb{N}$ such that $c_{j,k,\ell} = 0$ for all $k > K(\ell)$. Assume that for every $m \ge t$ the upper left $m \times m$ submatrices $C_1^{(m \times m)},\ldots,C_s^{(m \times m)}$ of $C_1,\ldots,C_s$, respectively, generate a digital $(t,m,s)$-net over $\FF_2$.

Consider the following construction principle for infinite sequences of points in $[0,1)^s$: represent $n \in \NN_0$ in base 2, say $n = n_0 + n_1 2 + \cdots + n_{m-1} 2^{m-1} \in \mathbb{N}_0$  with binary digits $n_0,n_1,\ldots \in \{0,1\}$, and multiply for every $j \in [s]$ the matrix $C_j$ with the infinite vector $\vec{n} = (n_0, n_1, \ldots, n_{m-1}, 0, 0, \ldots )^\top$ from $\mathbb{F}_2^{\mathbb{N}}$,
\begin{equation}\label{eq_dig_seq}
C_j \vec{n}=:(y_1^{(j)},y_2^{(j)},y_3^{(j)},\ldots)^\top,
\end{equation}
where the matrix vector product is evaluated over $\mathbb{F}_2$. Now set
\begin{equation*}
x_n^{(j)} =  \frac{y_1^{(j)}}{2}+\frac{y_2^{(j)}}{2^2} + \frac{y_3^{(j)}}{2^3}+\cdots
\end{equation*}
and $$\bsx_n=(x_n^{(1)},\ldots,x_n^{(s)}).$$
The infinite sequence $(\bsx_n)_{n \ge 0}$ is called a {\it digital $(t,s)$-sequence over $\FF_2$  with generating matrices $C_1,\ldots,C_s$}.
\end{definition}

Again we remark that explicit constructions for digital $(t,s)$-sequences are known with some restrictions on the parameter $t$, which depends on $s$.

For general properties of digital nets and sequences we refer to the books \cite{DP10,LP14,niesiam}. Explicit constructions of digital sequences over $\FF_2$ by Sobol, Niederreiter, Niederreiter-Xing and others are known, see again \cite{DP10,LP14,niesiam}.

\section{Quasi-Monte Carlo sampling of mixture product-distributions on intervals}\label{sec:QMCmix}

Assume we are given a ``non-normalized'' mixture distribution $\varphi: [\boldsymbol{a}, \boldsymbol{b}] \to \mathbb{R}_0^+$, where $\mathbb{R}_0^+ := \{x \in \mathbb{R}: x \ge 0\}$ and where $\boldsymbol{a} = (a_1, \ldots, a_s)$, $\boldsymbol{b} = (b_1, \ldots, b_s)$, $-\infty < a_j < b_j < \infty$, which can be expressed as a linear combination of products of one-dimensional PDFs $\varphi_{j, k_j}: [a_j, b_j] \to \mathbb{R}_{ 0}^+$ of the form\footnote{It would be more natural to consider a mixture distribution of the form $\sum_{\ell=1}^M c_\ell \prod_{j=1}^s \varphi_{j,\ell}(x_j)$. The results in this section apply for such mixtures in an analoguous manner. However, later on we consider mixture distributions which are hat function approximations based on weighted grids. In these instances, the notation used in this section is the natural way of writing it, and the current formulation in this section makes it easier to refer back to this section.}
\begin{equation}\label{tensform}
\varphi(\bsx) = \sum_{k_1=0}^{n_1} \ldots \sum_{k_s=0}^{n_s} c_{(k_1,\ldots,k_s)} \prod_{j=1}^s \varphi_{j, k_j}(x_j),
\end{equation}
where $n_1,\ldots,n_s \in \NN$, $c_{(k_1,\ldots, c_s)} \ge 0$ and
\begin{equation}\label{sum_cond_c}
\sum_{k_1=0}^{n_1} \ldots \sum_{k_s=0}^{n_s} c_{(k_1,\ldots,k_s)} = c > 0.
\end{equation}
This leads to a normalized PDF $\overline{\varphi} = \varphi/c$.

Let $\Phi_{j,k_j}$ be the cumulative distribution function (CDF) corresponding to $\varphi_{j,k_j}$ given by $$\Phi_{j,k_j}(x)=\int_{a_j}^x \varphi_{j,k_j}(y) \rd y.$$ We assume that the $\Phi_{j,k_j}: [a_j, b_j] \rightarrow [0, 1]$ are invertible and denote its inverse by   $\Phi_{j,k_j}^{-1}:[0, 1] \rightarrow [a_j, b_j]$. Let further
$\Phi_{(k_1,\ldots,k_s)}^{-1} = (\Phi_{1,k_1}^{-1},\ldots, \Phi_{s,k_s}^{-1})$ be defined as
$\Phi_{(k_1,\ldots,k_s)}^{-1}(\bsx) = (\Phi_{1,k_1}^{-1}(x_1),\ldots, \Phi_{s,k_s}^{-1}(x_s))$ for $\bsx = (x_1, \ldots, x_s) \in [0, 1]^s$. Note that as with $\Phi_{j,k_j}$, $\Phi_{j,k_j}^{-1}$ is also monotone and therefore also Borel-measurable.

For $n_1,\ldots,n_s \in \mathbb{N}$ we define the set of multi-indices $$\cN=\cN_{(n_1,\ldots,n_s)}:= \bigotimes_{j=1}^s \{0,1,\ldots,n_j\}.$$ For a multi-index $\bsk=(k_1,\ldots,k_s)$ in $\cN$, we write $c_{\bsk}=c_{(k_1,\ldots,k_s)}$ and $\varphi_{\bsk}(\bsx)=\prod_{j=1}^s \varphi_{j, k_j}(x_j)$. Since $c_{\bsk}/c \ge 0$ and $\sum_{\bsk \in \cN} (c_{\bsk}/c) = 1$, we can interpret the $c_{\bsk}/c$ as the probability that a random sample with law $\varphi$ comes from a sample with law $\varphi_{\bsk}$.

Now we construct our algorithm for $\overline{\varphi}$-weighted integration over $[\bsa,\bsb]$. Let $N \in \NN$, $N \ge 2$ be the total QMC sample size. Choose an ordering $\bsk_1, \bsk_2, \bsk_3, \ldots $ of the elements in $\cN$ such that $c_{\bsk_1} \ge c_{\bsk_2} \ge c_{\bsk_3} \ge \dots$, where for the case $c_{\bsk_i} = c_{\bsk_j}$ the ordering is arbitrary. Then for some threshold $\delta \in (0,N)$ let $r=r(\delta,N)$ be the smallest number such that 
\begin{equation}\label{def:r}
\sum_{v=1}^r \frac{c_{\bsk_v}}{c}  \ge 1-\frac{\delta}{N}
\end{equation}
and define the associated index subset
\begin{equation*}
\cL_\delta = \cL_\delta(N) = \left\{\bsk_1, \bsk_2, \ldots, \bsk_r\right\} \subseteq \cN.
\end{equation*}
Because of \eqref{sum_cond_c} it is clear that $r \le |\cN|=\prod_{j=1}^s(n_j+1)$.

Let $\cL_\delta^*=\cL_\delta\setminus\{\bsk_r\}$. Now, for $\bsk \in \cL_\delta^*$ define $N_{\bsk}$ as the largest integer that is less than or equal to $N c_{\bsk}/c$, i.e., $N_{\bsk}=\lfloor N c_{\bsk}/c \rfloor$. Note that $$\sum_{\bsk\in \cL_\delta^*} N_{\bsk}  =\sum_{\bsk \in \cL_\delta^*}\left\lfloor\frac{N c_{\bsk}}{c}\right\rfloor \le \frac{N}{c}  \sum_{\bsk \in \cL_\delta^*} c_{\bsk} \le N.$$ Let further  $$N_{\bsk_r}=N-\sum_{\bsk \in \cL_\delta^*} N_{\bsk}.$$
Then the following Diophantine approximation properties are satisfied:
\begin{itemize}
    \item $\sum_{\bsk \in \cL_\delta}N_{\bsk}=N$,
    \item for $\bsk \in \cL_{\delta}^{\ast}$ we have $$\left|\frac{c_{\bsk}}{c} - \frac{N_{\bsk}}{N}\right| \le \frac{1}{N},$$
    \item for $\bsk_r$ we have
    \begin{eqnarray*}
    \left|\frac{ c_{\bsk_r}}{c} -\frac{N_{\bsk_r}}{N}\right| & = & \left|\sum_{v=1}^r \frac{ c_{\bsk_v}}{c} - \sum_{v=1}^{r-1} \frac{ c_{\bsk_v}}{c} -\frac{1}{N}\left(N- \sum_{v=1}^{r-1} N_{\bsk_v} \right)\right|    \\
    & = & \left|\sum_{v=1}^r \frac{ c_{\bsk_v}}{c} -1 +\sum_{v=1}^{r-1}\left(\frac{N_{\bsk_v}}{N}-\frac{c_{\bsk_v}}{c}\right)  \right|\\
    & \le & 1 - \sum_{v=1}^r \frac{ c_{\bsk_v}}{c} +\frac{r-1}{N}\\
    & \le & \frac{\delta}{N} + \frac{r-1}{N}, 
    \end{eqnarray*}
    \item and, finally, 
    \begin{equation}\label{est:dioph3}
    \sum_{\bsk \in \cL_\delta} \left|\frac{c_{\bsk}}{c} - \frac{N_{\bsk}}{N}\right| \le  \frac{\delta+2(r-1)}{N}.
    \end{equation}
\end{itemize}

Let $(\bsy_n)_{n \ge 0}$ be a digital $(t,s)$-sequence over $\FF_2$ and let $N \in \NN$. Now we approximate the integral $$\int_{[\bsa,\bsb]} f(\bsx) \overline{\varphi}(\bsx) \rd \bsx$$ by the algorithm 
\begin{equation}\label{alg_A}
\mathcal{A}_{N,s}(f)= \frac{1}{c} \sum_{\bsk \in \cL_{\delta}} \frac{c_{\bsk}}{N_{\bsk}} \sum_{n=0}^{N_{\bsk}-1}f(\Phi_{\bsk}^{-1}(\bsy_n)).
\end{equation}
This is our algorithm for the integration over $[\bsa,\bsb]$ with respect to mixture distributions of the form \eqref{tensform}. 

Before we state the error estimate we introduce a shorthand notation that allows to formulate the error bound more conveniently. For an interval $[\bsa,\bsb]$, weights $\bsgamma=\{\gamma_{\uu}\}_{\uu \subseteq [s]}$ and $q \ge 1$, we set
\begin{equation}\label{def:Gab}
G_{\bsgamma,q,\bsa,\bsb}(N):=  \left(\sum_{\emptyset \neq \uu \subseteq [s]}  \left(\gamma_\uu \,   \prod_{j \in \uu} (3 (b_j - a_j) \log N)\right)^q \right)^{1/q}.
\end{equation}
Note that for fixed dimension $s$ the quantity $G_{\bsgamma,q,\bsa,\bsb}(N)$ is of order of magnitude $\mathcal{O}((\log N)^s)$ for $N$ growing to infinity. If the weights $\gamma_{\uu}$ are small enough, then we can get rid of the unfavorable dependence on the dimension $s$. For example, in case of product weights $\gamma_\uu = \prod_{j \in \uu} \gamma_j$, where $(\gamma_j)_{j \ge 1}$ is a sequence of positive reals, we have $$G_{\bsgamma,q,\bsa,\bsb}(N)=\left(-1+\prod_{j=1}^s\left(1+(\gamma_j 3 (b_j-a_j) \log N)^q\right)\right)^{1/q}.$$
Now, if $\sum_{j=1}^{\infty} (\gamma_j (b_j-a_j))^q < \infty$, then for arbitrary small $\rho>0$ there exists a $C(\rho)>0$, which is independent of the dimension $s$, such that $$G_{\bsgamma,q,\bsa,\bsb}(N) \le C(\rho) N^{\rho} \quad \mbox{for all $N \in \mathbb{N}$.}$$ This follows, for example, in the same way as \cite[Proof of Corollary~5.45]{DP10}.

Now we can state our first basic result.

\begin{theorem}\label{thm:main2}
Let $\overline{\varphi}$ be the normalized PDF of a mixture distribution like in \eqref{tensform}. Let $1\le p,q\le \infty$ be such that $1/p+1/q=1$. Let $N \in \NN$, let $\delta \in (0,N)$ and let $\mathcal{A}_{N,s}$ be the algorithm from \eqref{alg_A} based on a digital $(t,s)$-sequence $(\bsy_n)_{n \ge 0}$ over $\FF_2$ with non-singular upper triangular generating matrices $C_1,\ldots,C_s$.  Then for $f$ defined on $[\bsa,\bsb]$ with $\|f\|_{p,s,\bsgamma}<\infty$ we have
$$
\left|\int_{[\bsa,\bsb]} f(\bsx) \overline{\varphi}(\bsx) \rd \bsx - \mathcal{A}_{N,s}(f)\right| \le \|f\|_{p,s,\bsgamma} \frac{\delta+3 r-2}{N} \, \left[1 +  5\cdot 2^t \, \log (2N)  G_{\bsgamma,q,\bsa,\bsb}(N)\right],
$$
where $r=r(\delta,N)$ is given by \eqref{def:r} and depends on $\delta$ and $N$ (but $r$ is bounded from above by the number of mixture terms of $\overline{\varphi}$).
\end{theorem} 

\begin{remark}\rm
For fixed $N \in \NN$ the parameter $\delta \in (0,N)$ is a free parameter that should be chosen such that $\delta+3 r$ is small. For example, let $A:= \min \frac{c_{\bsk}}{c}$ where the minimum is extended over all $\bsk \in \cN_{(n_1,\ldots,n_s)}$. Then choose $$\delta = \frac{3N}{3+AN}.$$ With this choice \eqref{def:r} is satisfied with $$r= \left\lceil \frac{ N}{3+AN} \right\rceil $$ and then $$\delta+3 r  =   \frac{3N}{3+AN}  + 3 \left\lceil \frac{N}{3+AN} \right\rceil.$$ However, if $A$ is very small this might be prohibitive.
\end{remark}

For the proof of Theorem~\ref{thm:main2} we need the following result for densities of product form  which is a direct consequence of \cite[Theorem~3]{DP2020} (we replace the worst-case error with the integration error and multiply the bound by the norm of the integrand and we use $t = \max_{\emptyset \not=\uu \subseteq [s]} t_{\uu}$).

\begin{proposition}\label{thmDP1}
Let the density $\overline\varphi$ be of product form, i.e., $\overline{\varphi}(\bsx)=\prod_{j=1}^s \overline{\varphi}_j(x_j)$ with PDFs $\overline{\varphi}_j:[a_j,b_j]\rightarrow \RR$ for $j \in [s]$. Let $N \in \NN$ and let $\cP'_N$ be the initial segment of a digital $(t,s)$-sequence over $\FF_2$ consisting of the first $N$ terms. Assume that the generating matrices $C_1,\ldots,C_s$ are non-singular upper triangular matrices. Let $\cP_N=\{\Phi^{-1}(\bsy)\, : \, \bsy \in \cP'_N\}$, where $\Phi^{-1}$ is the inverse CDF corresponding to $\overline{\varphi}$. Let $1\le p,q\le \infty$ be such that $1/p+1/q=1$. Then for any $N \in \mathbb{N}$ with $N\ge 2$ we have
$$
\left|\int_{[\bsa,\bsb]} f(\bsx) \overline{\varphi}(\bsx) \rd \bsx - \frac{1}{N} \sum_{\bsx \in \cP_N} f(\bsx) \right|
\le \|f\|_{p,s,\bsgamma}  \frac{5\cdot 2^t\, \log (2N)}{N}   G_{\bsgamma,q,\bsa,\bsb}(N).
$$
\end{proposition}

Now we give the proof of Theorem~\ref{thm:main2}.

\begin{proof}[Proof of Theorem~\ref{thm:main2}]
Let $f$ be a function on $[\bsa,\bsb]$ with $\|f\|_{p,s,\bsgamma}< \infty$. Then we have
\begin{eqnarray*}
c \lefteqn{\left| \int_{[\bsa, \bsb]} f(\bsx) \overline{\varphi}(\bsx) \,\mathrm{d} \bsx - \mathcal{A}_{N,s}(f) \right|}\\
& = & \left|\sum_{\bsk \in \cN} c_{\bsk} \int_{[\bsa, \bsb]} f(\bsx) \varphi_{\bsk}(\bsx) \,\mathrm{d} \bsx - \sum_{\bsk \in \cL_\delta} \frac{c_{\bsk}}{N_{\bsk}} \sum_{n=0}^{N_{\bsk}-1}f(\Phi_{\bsk}^{-1}(\bsy_n)) \right| \\
& = & \left|\sum_{\bsk \in \cN\setminus \cL_\delta} c_{\bsk} \int_{[\bsa, \bsb]} f(\bsx) \varphi_{\bsk}(\bsx) \,\mathrm{d} \bsx + \sum_{\bsk \in \cL_\delta} c_{\bsk} \int_{[\bsa, \bsb]} f(\bsx) \varphi_{\bsk}(\bsx) \,\mathrm{d} \bsx - \sum_{\bsk \in \cL_\delta} \frac{c_{\bsk}}{N_{\bsk}} \sum_{n=0}^{N_{\bsk}-1}f(\Phi_{\bsk}^{-1}(\bsy_n)) \right| \\
& \le & \sum_{\bsk \in \cN\setminus \cL_\delta} c_{\bsk} \left| \int_{[\bsa, \bsb]} f(\bsx) \varphi_{\bsk}(\bsx) \,\mathrm{d} \bsx\right| + \sum_{\bsk \in \cL_\delta} c_{\bsk} \left|  \int_{[\bsa, \bsb]} f(\bsx) \varphi_{\bsk}(\bsx) \,\mathrm{d} \bsx - \frac{1}{N_{\bsk}} \sum_{n=0}^{N_{\bsk}-1}f(\Phi_{\bsk}^{-1}(\bsy_n)) \right| \\
& \le &  \|f\|_{p,s,\bsgamma} \left[\sum_{\bsk \in \cN\setminus \cL_\delta } c_{\bsk} + \sum_{\bsk \in \cL_\delta }  \frac{c_{\bsk}}{N_{\bsk}} 5\cdot 2^t\, \log (2N_{\bsk})   G_{\bsgamma,q,\bsa,\bsb}(N_{\bsk})  \right],
\end{eqnarray*}
where we used Proposition~\ref{thmDP1} and the fact that $$\sup_{\bsx \in [\bsa,\bsb]}|f(\bsx)| \le \|f\|_{p,s,\bsgamma}$$ and therefore also  $$\left| \int_{[\bsa, \bsb]} f(\bsx) \varphi_{\bsk}(\bsx) \,\mathrm{d} \bsx\right| \le \|f\|_{p,s,\bsgamma} \int_{[\bsa, \bsb]}  \varphi_{\bsk}(\bsx) \,\mathrm{d} \bsx =  \|f\|_{p,s,\bsgamma}.$$ 
Using $N_{\bsk}\le N$ we obtain
\begin{eqnarray*}
c\lefteqn{\left| \int_{[\bsa, \bsb]} f(\bsx) \overline{\varphi}(\bsx) \,\mathrm{d} \bsx - \mathcal{A}_{N,s}(f) \right|}\\
& \le & \|f\|_{p,s,\bsgamma} \left[\sum_{\bsk \in \cN\setminus \cL_\delta} c_{\bsk} +  5 \cdot 2^t\,\log (2 N) \, G_{\bsgamma,q,\bsa,\bsb}(N)  \sum_{\bsk \in \cL_\delta }  \frac{c_{\bsk}}{N_{\bsk}} \right].
\end{eqnarray*}
We have 
$$
\sum_{\bsk \in \cN\setminus \cL_\delta} c_{\bsk} \le  \frac{c \delta}{N}
$$
and, using \eqref{est:dioph3}, %
\begin{eqnarray*}
\sum_{\bsk \in \cL_\delta}  \frac{c_{\bsk}}{N_{\bsk}}  & \le &\sum_{\bsk \in \cL_\delta}  \left(\frac{c_{\bsk}}{N_{\bsk}}-\frac{c}{N}\right) + \frac{c\, |\cL_\delta|}{N}  \\
& = & \sum_{\bsk \in \cL_\delta}  \frac{c}{N_{\bsk}} \left(\frac{c_{\bsk}}{c}-\frac{N_{\bsk}}{N}\right) + \frac{c\, r}{N} \\
& \le & c\, \frac{\delta+2(r-1)}{N}  + \frac{c\, r}{N}\\
& \le & c\, \frac{\delta+3 r-2}{N}.
\end{eqnarray*}
Therefore we obtain
\begin{eqnarray*}
\lefteqn{\left| \int_{[\bsa, \bsb]} f(\bsx) \overline{\varphi}(\bsx) \,\mathrm{d} \bsx - \mathcal{A}_{N,s}(f) \right|}\\
& \le & \|f\|_{p,s,\bsgamma} \left[\frac{\delta}{N} +  5 \cdot 2^t\, \log(2 N)\,  G_{\bsgamma,q,\bsa,\bsb}(N)  \frac{\delta+3 r-2}{N}  \right]\\
& \le &  \|f\|_{p,s,\bsgamma} \frac{\delta+3 r-2}{N} \, \left[1 +  5 \cdot 2^t \, \log(2  N) \,  G_{\bsgamma,q,\bsa,\bsb}(N)\right],
\end{eqnarray*}
as desired.
\end{proof}

\section{Quasi-Monte Carlo sampling of hat function approximations of bounded target densities defined on intervals}\label{sec4:QMChat}

In this section we use hat function approximations of a given target density defined on an interval. The hat function approximation can be viewed as a mixture distribution as discussed in the previous section, where each PDF is a hat function. We consider also adaptive hat function approximations in this section.

The results on hat function approximations are well understood, we include them here for completeness and to fix the notation.

\subsection{Piecewise linear hat function approximation on intervals}\label{subsec:hatfkt}

Let $m \in \mathbb{N}$ and 
\begin{equation}\label{def:yl}
y_\ell := a+ \ell \, \frac{b-a}{m} \quad \mbox{for $\ell \in \{0, 1, \ldots, m\}$.}
\end{equation}
Obviously, $y_0=a$ and $y_m=b$. For $\ell \in \{0, 1,\ldots, m\}$ define the piecewise linear hat function $h_{\ell, m}: [a, b] \to \mathbb{R}$ to be $1$ at $y_\ell$, to be $0$ at $y_{\ell-1}$ and $y_{\ell+1}$ and also $0$ outside the interval $[y_{\ell-1}, y_{\ell+1}]$, and linear in between, that is, define
$$h_{0,m}(x):=\left\{
\begin{array}{ll}
(y_{1}-x) \frac{ m}{b-a} & \mbox{ if } y_0 \le x \le y_{1},\\[1em]
0 & \mbox{ otherwise},
\end{array}\right.$$
for $\ell \in \{1,\ldots,\ell-1\}$ define 
$$h_{\ell,m}(x):=\left\{
\begin{array}{ll}
(x-y_{\ell-1}) \frac{ m}{b-a} & \mbox{ if } y_{\ell-1} \le x \le y_\ell,\\[1em]
(y_{\ell+1}-x) \frac{ m}{b-a} & \mbox{ if } y_{\ell} \le x \le y_{\ell+1},\\[1em]
0 & \mbox{ otherwise},
\end{array}\right.$$
and
$$h_{m,m}(x):=\left\{
\begin{array}{ll}
(x-y_{m-1}) \frac{m}{b-a} & \mbox{ if } y_{m-1} \le x \le y_m,\\[1em]
0 & \mbox{ otherwise}.
\end{array}\right.$$
For $x \in [a,b]$ let $k \in \{0,1,\ldots,m-1\}$ be such that $x \in [y_k,y_{k+1})$. Then we have 
\begin{align*}
\sum_{\ell=0}^m h_{\ell,m}(x)= & h_{k,m}(x)+h_{k+1,m}(x)\\ 
= & (y_{k+1}-x) \frac{m}{b-a} +(x-y_k) \frac{m}{b-a} =(y_{k+1}-y_k) \frac{m}{b-a} = 1.
\end{align*}

In dimension $s > 1$ consider $[\bsa,\bsb]$ where $\bsa=(a_1,\ldots,a_s)$ and $\bsb=(b_1,\ldots,b_s)$ are from $\mathbb{R}^s$. For $\bsm=(m_1,\ldots,m_s) \in \mathbb{N}^s$ and $\bsell=(\ell_1,\ldots,\ell_s) \in M_{\bsm} := \bigotimes_{j=1}^s \{0,1,\ldots,m_j\}$ we use $h_{\bsell, \bsm} := \prod_{j=1}^s h_{\ell_j, m_j}$. Note that for $\bsx=(x_1,\ldots,x_s)\in [\bsa,\bsb]$ we have
\begin{eqnarray*}
\sum_{\bsell \in M_{\bsm}} h_{\bsell,\bsm}(\bsx) =  \prod_{j=1}^s \left(\sum_{\ell =0}^{m_j} h_{\ell,m_j}(x_j)\right)=  1. %
\end{eqnarray*}

Consider a general PDF $\overline{\pi}$ and assume we can only evaluate its unnormalized density $\pi$. Then, we first approximate $\pi$ by
\begin{equation*}
\varphi_{\bsm}(\bsx) = \sum_{\bsell \in M_{\bsm}} \pi(\bsy_{\bsell}) h_{\bsell, \bsm}(\bsx) \quad\mbox{for $\bsx \in [\bsa,\bsb]$,}
\end{equation*}
where $\bsm=(m_1,\ldots,m_s)\in \NN^s$ and where $y_{\bsell}$ denotes the vector $(y_{1,\ell_1},\ldots,y_{s,\ell_s})$ (with resolution $m_j$ and with respect to the interval $[a_j,b_j]$ in coordinate $j \in [s]$). The function $\varphi_{\bsm}$ is a piecewise linear approximation of $\pi$. The following lemma gives an error estimate for this kind of approximation. 

\begin{lemma}\label{le:app1hat}
Assume that $\pi$ satisfies a H\"older condition 
\begin{equation}\label{Hcond}
|\pi(\bsx)-\pi(\bsy)| \le L_{\pi} \|\bsx-\bsy\|_{\infty}^{\beta} \quad \mbox{for all $\bsx,\bsy \in [\bsa,\bsb]$,}
\end{equation}
with H\"older constant $L_{\pi}>0$ and for some exponent $\beta \in (0,1]$.
Then for every $\bsx \in [\bsa,\bsb]$ we have
$$|\pi(\bsx) - \varphi_{\bsm}(\bsx) | \le L_{\pi}\max_{j \in [s]} \left(\frac{b_j-a_j}{m_j}\right)^{\beta}.$$
\end{lemma}

\begin{proof}
Suppose that $\bsx$ belongs to the subinterval $\bigotimes_{j=1}^s [y_{j, k_j},y_{j, k_j+1})$ of $[\bsa,\bsb]$. Then the error is
\begin{eqnarray*}
\lefteqn{|\pi(\bsx) - \varphi_{\bsm}(\bsx)|}\\
& = & \left|\sum_{\bsell \in M_{\bsm}} (\pi(\bsx)-\pi(\bsy_{\bsell})) h_{\bsell,\bsm}(\bsx)\right|\\
& = & \left|\sum_{\bsell \in \bigotimes_j \{k_j,k_{j+1}\}} (\pi(\bsx)-\pi(\bsy_{\bsell})) h_{\bsell,\bsm}(\bsx)\right|\\
& = & \left|\sum_{\uu \subseteq [s]} (\pi(\bsx)-\pi((y_{j,k_j})_{\uu},(y_{j,k_j+1})_{\uu^c})) \prod_{j\in \uu} \frac{(y_{j,k_j+1}-x_j)m_j}{b_j-a_j} \prod_{j \not\in \uu} \frac{(x_j-y_{j,k_j})m_j}{b_j-a_j} \right|\\
& \le & L_{\pi}\max_{j \in [s]} \left(\frac{b_j-a_j}{m_j}\right)^{\beta} \sum_{\uu \subseteq [s]}\prod_{j\in \uu} \frac{(y_{j, k_j+1}-x_j)m_j}{b_j-a_j} \prod_{j \not\in \uu} \frac{ (x_j-y_{j, k_j})m_j}{b_j-a_j} \\
& = & L_{\pi}\max_{j \in [s]} \left(\frac{b_j-a_j}{m_j}\right)^{\beta} \prod_{j=1}^s \frac{(y_{j, k_j+1}-y_{j, k_j}) m_j}{b_j-a_j}\\
& = & L_{\pi}\max_{j \in [s]} \left(\frac{b_j-a_j}{m_j}\right)^{\beta},
\end{eqnarray*}
where we used the H\"older condition \eqref{Hcond}.
\end{proof}

Let now $\pi: [\bsa, \bsb] \to \mathbb{R}$ be a bounded target density with
\begin{equation*}
\|\pi \|_{\infty} := \sup_{\bsx \in [\bsa, \bsb] } \left| \pi(\bsx) \right| < \infty.
\end{equation*}
Let the hat function approximation of $\pi$ be given by
\begin{equation*}
\varphi_{\bsm}(\bsx) = \sum_{\ell_1 = 0}^{m_1} \ldots \sum_{\ell_s=0}^{m_s} \pi(y_{1,\ell_1}, \ldots, y_{s,\ell_s}) h_{\bsell, \bsm}(\bsx) \quad \mbox{for $\bsx \in [\bsa,\bsb]$,}
\end{equation*}
where $\bsm=(m_1,\ldots,m_s)\in \NN^s$ and where the $y_{j,\ell_j}$ are like in \eqref{def:yl} with respect to resolution $m_j$ and interval $[a_j,b_j]$ in coordinate direction $j\in [s]$. This approximation guarantees that all the coefficients in the approximation are non-negative since $\pi$ is non-negative by assumption.

We also assume that the target density satisfies the H\"older condition \eqref{Hcond}. Then, according to Lemma~\ref{le:app1hat}, we have
\begin{equation}\label{est_hold1}
|\pi(\bsx) - \varphi_{\bsm}(\bsx)| \le L_{\pi} \max_{j \in [s]} \left(\frac{b_j-a_j}{m_j}\right)^{\beta}\quad \mbox{for all $\bsx \in [\bsa,\bsb]$.}
\end{equation}
Note that $\pi \ge 0$ implies that $\varphi_{\bsm} \ge 0$, however, $\varphi_{\bsm}$ is in general not a probability distribution since we cannot guarantee that $\int_{[\bsa,\bsb]} \varphi_{\bsm}(\bsx) \rd \bsx = 1$. The hat functions are normalized such that the value at the peak is $1$. Corresponding to $h_{\ell, m}$ we define a probability density function $\varphi_{\ell, m}$ as
\begin{equation*}
\varphi_{\ell,m} = \frac{m}{b-a} h_{\ell,m} \quad \mbox{for $m \in \mathbb{N}$ and $\ell \in \{0,1,\ldots,m\}$.}
\end{equation*}
Thus we can write
\begin{equation*}
\varphi_{\bsm}(\bsx) = \sum_{\bsk \in M_{\bsm}} \frac{ \pi(\bsy_{\bsk}) }{\prod_{j=1}^s m_j/(b_j-a_j) } \prod_{j=1}^s \varphi_{k_j,m_j}(x_j),
\end{equation*}
where $\varphi_{k_j, m_j}$ are probability density functions and $\bsx=(x_1,\ldots,x_s)$ in $[\bsa,\bsb]$.

Now $\varphi_{\bsm}$ is of the form \eqref{tensform} with 
\begin{equation}\label{phi:mix}
c_{\bsk}=\left(\prod_{j=1}^s \frac{b_j-a_j}{ m_j }\right) \, \pi(\bsy_{\bsk})\quad \mbox{and}\quad \cN=M_{\bsm}
\end{equation}
and we can use Theorem~\ref{thm:main2}. Equation~\eqref{sum_cond_c} now reads
\begin{equation}\label{defc}
c = \left(\prod_{j=1}^s \frac{b_j-a_j}{ m_j }\right) \sum_{\bsk \in M_{\bsm}}  \pi(\bsy_{\bsk})  > 0.
\end{equation}
Thus, $c$ is an approximation of $z:=\int_{[\bsa,\bsb]} \pi(\bsx) \,\mathrm{d} \bsx$.

Now we consider the normalized probability density $\overline{\pi}:=\pi/z$. The normalized version of the H\"older condition \eqref{Hcond} for $\pi$ is then re-stated in the form
\begin{equation}\label{Hcondno}
|\overline{\pi}(\bsx)-\overline{\pi}(\bsy)| \le L_{\overline{\pi}} \|\bsx-\bsy\|_{\infty}^{\beta} \quad \mbox{for all $\bsx,\bsy \in [\bsa,\bsb]$,}
\end{equation}
for some exponent $\beta \in (0,1]$, where $L_{\overline{\pi}} := L_{\pi}/z$ is the H\"older constant for $\overline{\pi}$.

\begin{lemma}\label{lem_pi_approx}
Let $\overline{\pi}$ be the normalized probability density satisfying the H\"older condition \eqref{Hcondno} and let $\overline{\varphi}_{\bsm}$ be the normalized version of $\varphi_{\bsm}$. Then we have
$$\int_{[\bsa,\bsb]} | \overline{\pi}(\bsx) - \overline{\varphi}_{\bsm}(\bsx)| \rd \bsx \leq 2 L_{\overline\pi}  \prod_{j=1}^s (b_j-a_j) \max_{j \in [s]} \left( \frac{b_j-a_j}{m_j} \right)^\beta.$$
\end{lemma}

\begin{proof}

Recall that $\overline{\pi} = \pi/z$ and $\overline{\varphi}_{\bsm} = \varphi_{\bsm} / c$, where $c= \int_{[\bsa, \bsb]} \varphi_{\bsm}(\bsx) \,\mathrm{d} \bsx$. Then, we have 
\begin{align*}
\int_{[\bsa,\bsb]} | \overline{\pi}(\bsx) - \overline{\varphi}_{\bsm}(\bsx)| \rd \bsx \leq & \frac{1}{z} \int_{[\bsa,\bsb]} |\pi(\bsx) - \varphi_{\bsm}(\bsx)| \rd \bsx + \frac{1}{z} \int_{[\bsa,\bsb]} \left| \varphi_{\bsm}(\bsx) - \frac{z}{c} \varphi_{\bsm}(\bsx) \right| \rd \bsx   \\
= &\frac{1}{z} \int_{[\bsa,\bsb]} |\pi(\bsx) - \varphi_{\bsm}(\bsx)| \rd \bsx + \frac{|c - z|}{z} \int_{[\bsa,\bsb]} \frac{\varphi_{\bsm}(\bsx)}{c}  \rd \bsx
\\ \le & \frac2z \int_{[\bsa,\bsb]} | \pi(\bsx) - \varphi_{\bsm}(\bsx)| \rd \bsx \\ 
\le & \frac2z \prod_{j=1}^s (b_j-a_j) L_\pi \max_{j \in [s]} \left( \frac{b_j-a_j}{m_j} \right)^\beta\\ 
= & 2 L_{\overline{\pi}} \prod_{j=1}^s (b_j-a_j)  \max_{j \in [s]} \left( \frac{b_j-a_j}{m_j} \right)^\beta,
\end{align*}
where the second to last line above follows from Lemma~\ref{le:app1hat}.
\end{proof}

\subsection{Quasi-Monte Carlo sampling of the hat function approximation of a bounded target density defined on an interval}

The hat function approximation of a target density defined on an interval can be viewed as a mixture product-density approximation of the target density. Hence we can sample from the hat function approximation using the approach from Section~\ref{sec:QMCmix}.

We now use $\mathcal{A}_{N,s}$ defined in \eqref{alg_A} with $c_{\bsk}$ and $\cN$ like in \eqref{phi:mix}.

\begin{theorem}\label{thm_hat_approx_QMC}
Let $1\le p,q\le \infty$ be such that $1/p+1/q=1$. Assume that the target probability density  $\overline{\pi}:[\bsa,\bsb] \rightarrow \RR_0^+$ is bounded and satisfies the H\"older condition \eqref{Hcondno} with H\"older constant $L_{\overline{\pi}}$ and exponent $\beta \in (0,1]$, and assume that $f:[\bsa,\bsb]\rightarrow \RR$ satisfies $\|f\|_{p,s,\bsgamma}< \infty$, where the latter norm is given by \eqref{def:norm}. Let $\delta>0$ and let $\mathcal{A}_{N,s}$ be given by \eqref{alg_A} with $c_{\bsk}$ and $\cN$ like in \eqref{phi:mix} and based on a digital $(t,s)$-sequence over $\FF_2$ with non-singular upper triangular generating matrices $C_1,\ldots,C_s$.
Then we have
\begin{align*}
\lefteqn{\left|  \int_{[\bsa, \bsb]} f(\bsx) \overline{\pi}(\bsx) \rd \bsx - \mathcal{A}_{N,s}(f) \right|}\\ & \le 2 \|f\|_{L_{\infty}} L_{\overline{\pi}} \prod_{j=1}^s (b_j-a_j) \max_{j \in [s]} \left( \frac{b_j-a_j}{m_j} \right)^\beta  \\ 
&\quad + \|f\|_{p, s, \bsgamma} \frac{\delta+3r-2 }{N} \,  \left[1+ 5 \cdot 2^t \log(2 N) \, G_{\bsgamma,q,\bsa,\bsb}(N)\right],
\end{align*}
where $r=r(\delta,N)$ is given by \eqref{def:r} and $G_{\bsgamma,q,\bsa,\bsb}(N)$ is defined in \eqref{def:Gab}. 
\end{theorem}

\begin{proof}
Note that $\|f\|_{p,s,\bsgamma}< \infty$ implies $\|f\|_{L_{\infty}}< \infty$. We have the approximation in the following steps
\begin{eqnarray*}
\lefteqn{\int_{[\bsa, \bsb]} f(\bsx) \overline{\pi}(\bsx) \rd \bsx - \mathcal{A}_{N,s}(f)} \\ 
& = & \int_{[\bsa, \bsb]} f(\bsx) ( \overline{\pi}(\bsx) - \overline{\varphi}_{\bsm}(\bsx) ) \rd \bsx  +  \int_{[\bsa, \bsb]} f(\bsx) \overline{\varphi}_{\bsm}(\bsx) \rd \bsx - \mathcal{A}_{N,s}(f),
\end{eqnarray*}
where $\overline{\varphi}_{\bsm}$ is an approximation of $\pi$, which has then subsequently been normalized so that $\overline{\varphi}_{\bsm}$ is a probability density.

For the first part of the estimate we have from Lemma~\ref{lem_pi_approx}
\begin{equation}\label{intfphi}
\left| \int_{[\bsa, \bsb]} f(\bsx) ( \overline{\pi}(\bsx) - \overline{\varphi}_{\bsm}(\bsx) ) \rd \bsx \right| \le 2 \|f\|_{L_{\infty}} L_{\overline\pi}  \prod_{j=1}^s (b_j-a_j) \max_{j \in [s]} \left( \frac{b_j-a_j}{m_j} \right)^\beta.
\end{equation}

For the second part of the estimate, we have by Theorem~\ref{thm:main2} that
\begin{align*}
\lefteqn{\left| \int_{[\bsa, \bsb]} f(\bsx) \overline{\varphi}_{\bsm}(\bsx) \rd \bsx - \mathcal{A}_{N,s}(f) \right|} \\ 
& \le \|f\|_{p, s, \bsgamma} \frac{\delta+3r-2}{N} \,  \left[1+ 5 \cdot 2^t \log(2 N) \, G_{\bsgamma,q,\bsa,\bsb}(N)\right].
\end{align*}
This finishes the proof.
\end{proof}

\subsection{Coordinate-wise adaptive hat function approximation}\label{sec:3.2}

We discuss now an adaptive approximation of the target density using hat functions. First we generalize the definition of the hat functions in Section~\ref{subsec:hatfkt} to arbitrary partitions of the interval $[a, b]$ which allows us to use local refinements. 
Let $K \in \mathbb{N}$. Assume we are given a set of reals $y_k$ for $k \in \{0,1,\ldots,K\}$, such that $a = y_0 < y_1 < \cdots < y_{K-1} < y_K = b$ (previously we used an equally spaced partition with spacing $1/m$). Corresponding to this partition of $[a,b]$ we define the hat functions on $[a,b]$ by
$$h_{0}(x):=\left\{
\begin{array}{ll}
\frac{y_1 - x }{ y_1 - y_0 }  & \mbox{ if } y_0 \le x \le y_1,\\[1em]
0 & \mbox{ otherwise},
\end{array}\right.$$
for $k \in \{1,\ldots, K-1\}$ we define
$$h_{k}(x):=\left\{
\begin{array}{ll}
\frac{x-y_{k-1}}{y_k - y_{k-1}} & \mbox{ if } y_{k-1} \le x \le y_k,\\[1em]
\frac{y_{k+1}-x}{ y_{k+1} - y_k}  & \mbox{ if } y_k \le x \le y_{k+1},\\[1em]
0 & \mbox{ otherwise},
\end{array}\right.$$
and
$$h_{K}(x):=\left\{
\begin{array}{ll}
\frac{x-y_{K-1}}{y_K - y_{K-1}}  & \mbox{ if } y_{K-1}  \le x \le y_K,\\[1em]
0 & \mbox{ otherwise}.
\end{array}\right.$$
For $x \in [0,1)$ let $k \in \{0,1,\ldots, K-1\}$ be such that $x \in [y_k,y_{k+1})$. Then we have $$\sum_{\ell=0}^K h_{\ell}(x) = h_{k}(x)+h_{k+1}(x) = 1.$$

Consider now dimension $s > 1$ and an interval $[\bsa,\bsb]$. We assume that for every coordinate direction $j \in [s]$ we have a set of $y^{(j)}_k$ for $k \in \{0,1,\ldots,K_j\}$, such that
\begin{equation*}
a_j = y^{(j)}_0 < y^{(j)}_1 < \cdots < y^{(j)}_{K_j-1} < y^{(j)}_{K_j} = b_j.
\end{equation*}
Corresponding to this partition we define the hat functions $h_{k_j, j}$ for $k_j \in \{1, 2, \ldots, K_j\}$ and $j \in [s]$, as in the one-dimensional case
$$h_{0,j}(x):=\left\{
\begin{array}{ll}
\frac{y^{(j)}_{1} - x }{ y^{(j)}_{1} - y^{(j)}_{0} }  & \mbox{ if } y^{(j)}_{0} \le x \le y^{(j)}_{1},\\[1em]
0 & \mbox{ otherwise},
\end{array}\right.$$
for $k_j \in \{1,\ldots,K_j-1\}$
$$h_{k_j, j}(x):=\left\{
\begin{array}{ll}
\frac{x-y^{(j)}_{k_j-1}}{y^{(j)}_{k_j} - y^{(j)}_{k_j-1}} & \mbox{ if } y^{(j)}_{k_j-1} \le x \le y^{(j)}_{k_j},\\[1em]
\frac{ y^{(j)}_{k_j+1}-x}{ y^{(j)}_{k_j+1} - y^{(j)}_{k_j}}  & \mbox{ if } y^{(j)}_{k_j} \le x \le y^{(j)}_{k_j+1},\\[1em]
0 & \mbox{ otherwise},
\end{array}\right.$$
and
$$h_{K_j, j}(x):=\left\{
\begin{array}{ll}
\frac{x-y^{(j)}_{K_j-1}}{y^{(j)}_{K_j} - y^{(j)}_{K_j-1}}  & \mbox{ if } y^{(j)}_{K_j-1} \le x \le y^{(j)}_{K_j},\\[1em]
0 & \mbox{ otherwise}.
\end{array}\right.$$
For $\bsk = (k_1, \ldots, k_s) \in \cK := \bigotimes_{j=1}^s \{0, 1, \ldots, K_j\}$ we define the multivariate hat function by the product
\begin{equation*}
h_{\bsk}(\bsx) := \prod_{j=1}^s h_{k_j,j}(x_j).
\end{equation*}
Note that for $\bsx \in [\bsa,\bsb]$ we have
\begin{eqnarray*}
\sum_{\bsk \in \cK} h_{\bsk}(\bsx) =  \prod_{j=1}^s \left(\sum_{k_j =0}^{K_j} h_{k_j,j}(x_j)\right)=  1. 
\end{eqnarray*}
For $\bsk = (k_1, k_2, \ldots, k_s) \in \cK$, defining the grid points
\begin{equation}\label{gridpts_general}
\bsy_{\bsk} = ( y^{(1)}_{k_1}, y^{(2)}_{k_2}, \ldots, y^{(s)}_{k_s} ),
\end{equation}
we approximate $\pi$ by a piecewise linear approximation in the form of 
\begin{equation*}
\varphi_{\cK}(\bsx) = \sum_{\bsk \in \cK} \pi(\bsy_{\bsk}) h_{\bsk}(\bsx).
\end{equation*}

In the following, for an interval 
\begin{equation}\label{def:Rk}
R_{\bsk} := \bigotimes_{j=1}^s \left[y^{(j)}_{k_j},y^{(j)}_{k_j+1}\right)
\end{equation}
for $\bsk\in \cK^\ast := \bigotimes_{j=1}^s \{0, 1, \ldots, K_j-1\}$ we will denote the diameter (in $\ell_2$-norm) by ${\rm diam}(R_{\bsk})$, that is $${\rm diam}(R_{\bsk}):=\left(\sum_{j=1}^s (y^{(j)}_{k_j+1}-y^{(j)}_{k_j})^2\right)^{1/2}.$$

\begin{lemma}\label{le_adapt_pi_holder}
Assume that $\pi$ satisfies a local H\"older condition with exponent $\beta\in (0,1]$ of the form
\begin{equation}\label{Hcond2}
|\pi(\bsx)-\pi(\bsy)| \le L_{\pi}(\bsk) \, {\rm diam}([\bsx,\bsy])^{\beta} \quad \mbox{for all $\bsx, \bsy \in R_{\bsk}$,}
\end{equation}
for all $\bsk\in \cK^\ast $ and $R_{\bsk}$ like in \eqref{def:Rk}. Then, for $\bsx \in R_{\bsk}$ we have
$$|\pi(\bsx) - \varphi_{\cK}(\bsx) | \le L_{\pi}(\bsk) \, {\rm diam}(R_{\bsk})^{\beta} .$$
\end{lemma}

\begin{proof}
Assume that  $\bsx \in R_{\bsk}$ for some $\bsk \in \cK^\ast$. Then the error is
\begin{eqnarray*}
\lefteqn{ |\pi(\bsx) - \varphi_{\cK}(\bsx) | } \\ & = & \left|\sum_{\bsell \in \cK} (\pi(\bsx)-\pi(\bsy_{\bsell})) h_{\bsell}(\bsx)\right|\\
& = & \left|\sum_{\bsell \in \bigotimes_j \{k_j, k_{j}+1\}} (\pi(\bsx)-\pi(\bsy_{\bsell})) h_{\bsell}(\bsx)\right|\\
& = & \left|\sum_{\uu \subseteq [s]} \left(\pi(\bsx)-\pi((y^{(j)}_{k_j})_{\uu},(y^{(j)}_{k_j+1})_{\uu^c})\right) \prod_{j\in \uu} \frac{ y^{(j)}_{k_j+1}-x_j}{y^{(j)}_{k_j+1}-y^{(j)}_{k_j}} \prod_{j \not\in \uu} \frac{x_j-y^{(j)}_{k_j} }{y^{(j)}_{k_j+1}-y^{(j)}_{k_j}} \right|.
\end{eqnarray*}
Using \eqref{Hcond2} and observing that
$$
\sum_{\uu \subseteq [s]} \prod_{j\in \uu} \frac{ y^{(j)}_{k_j+1}-x_j}{y^{(j)}_{k_j+1}-y^{(j)}_{k_j}} \prod_{j \not\in \uu} \frac{x_j-y^{(j)}_{k_j} }{y^{(j)}_{k_j+1}-y^{(j)}_{k_j}} = \prod_{j=1}^s\left( \frac{ y^{(j)}_{k_j+1}-x_j}{y^{(j)}_{k_j+1}-y^{(j)}_{k_j}} + \frac{x_j-y^{(j)}_{k_j} }{y^{(j)}_{k_j+1}-y^{(j)}_{k_j}}\right)=1
$$
we obtain for $\bsx \in R_{\bsk}$ that
$$|\pi(\bsx) - \varphi_{\cK}(\bsx) | \le L_{\pi}(\bsk) \, {\rm diam}(R_{\bsk})^{\beta},$$
as desired.
\end{proof}

Again we switch to the normalized probability density $\overline{\pi}:=\pi/z$. The normalized version of the local H\"older condition \eqref{Hcond2} for $\pi$ is then re-stated in the form
\begin{equation}\label{Hcond2no}
|\overline{\pi}(\bsx)-\overline{\pi}(\bsy)| \le  L_{\overline{\pi}}(\bsk) \, {\rm diam}([\bsx,\bsy])^{\beta} \quad \mbox{for all $\bsx, \bsy \in R_{\bsk}$,}
\end{equation}
for some $\beta \in (0,1]$, for all $\bsk\in \cK^\ast$, where $R_{\bsk}$ is like in \eqref{def:Rk} and where $L_{\overline{\pi}}(\bsk) := L_{\pi}(\bsk)/z$ are the local H\"older constants for $\overline{\pi}$.

The proof of the following lemma is analoguous to the proof of Lemma~\ref{lem_pi_approx}, where we use Lemma~\ref{le_adapt_pi_holder} instead of Lemma~\ref{le:app1hat}.

\begin{lemma}\label{lem_pi_adapt}
Let $\overline{\pi}$ be the normalized probability density satisfying the local H\"older condition \eqref{Hcond2no} and let $\overline{\varphi}_{\cK}$ be the normalized version of $\varphi_{\cK}$. Then we have
$$\int_{[\bsa,\bsb]} | \overline{\pi}(\bsx) - \overline{\varphi}_{\cK}(\bsx)| \rd \bsx \leq 2  \sum_{\bsk = (k_1, \ldots, k_s) \in \cK^\ast} \prod_{j=1}^s (y_{k_j+1}^{(j)} - y_{k_j}^{(j)}) L_{\overline{\pi}}(\bsk) \, {\rm diam}(R_{\bsk})^{\beta}.$$
\end{lemma}

Hence the goal of the adaptive strategy is to find a partitioning such that the terms $$\prod_{j=1}^s (y_{k_j+1}^{(j)} - y_{k_j}^{(j)}) L_{\overline{\pi}}(\bsk) \, {\rm diam}(R_{\bsk})^{\beta} $$ all have a similar value. Roughly speaking, we refine in those regions where $\prod_{j=1}^s (y_{k_j+1}^{(j)} - y_{k_j}^{(j)}) L_{\overline{\pi}}(\bsk) \, {\rm diam}(R_{\bsk})^{\beta} $ is large.

In practice, we construct the approximation $\varphi_{\cK}(\bsx)$ iteratively. We label each coordinate interval $[y_k^{(j)}, y_{k+1}^{(j)})$ for $k \in \{0, \ldots, K_j-1\}$ and $j \in [s]$ with a binary variable $\theta_{k}^{(j)} \in \{\mathrm{true,false}\}$, which indicates if the coordinate interval $[y_k^{(j)}, y_{k+1}^{(j)})$ will be refined. In each iteration, we first refine every coordinate interval that is labelled as ``true'' by dividing it into two equal-size sub-intervals. This gives a new grid. Then, by evaluating the function $\pi$ on the newly added grid points, we build an intermediate approximation, denoted by $\varphi_{\cK'}(\bsx)$. In the following adaptive refinement, we only want to keep the grid points in $\varphi_{\cK'}(\bsx)$ that contribute to the error reduction.

The difference between $\varphi_{\cK'}(\bsx)$ and $\varphi_{\cK}(\bsx)$ provides local error indicators for the original approximation $\varphi_{\cK}(\bsx)$. For each of the coordinate intervals on the original grid that contributes to any of the local error indicators above the prescribed refinement threshold, we refine it and label the refined sub-intervals as ``true'' for future refinement candidates; otherwise, we keep the original interval and label it ``false''. This way, we can adaptively refine the approximation in locations where the error is large. After each step of adaptive refinement, the new approximation is obtained without any extra function evaluations, as the set of grid points used is a subset of the grid points in $\varphi_{\cK'}(\bsx)$. We stop the adaptive refinement until all coordinate intervals are labelled as ``false''.

\subsection{Quasi-Monte Carlo sampling of the adaptive hat function approximation of a bounded target density defined on an interval}

We can generalize Theorem~\ref{thm_hat_approx_QMC} to the adaptive hat function approximation $\varphi_{\cK}$ introduced in the previous section. At first we need to normalize. The normalization constants are
\begin{equation}\label{norm:const}
c_{k_j}^{(j)}=\begin{cases}
\frac{2}{y^{(j)}_{1}-a_j} & \mbox{for } k_j=0,\\[0.5em]
\frac{2}{y^{(j)}_{k_j+1}-y^{(j)}_{k_j-1}} & \mbox{for } k_j \in \{1,\ldots,K_j-1\},\\[0.5em]
\frac{2}{b_j-y^{(j)}_{K_j-1}} & \mbox{for } k_j = K_j. 
\end{cases}
\end{equation} 
We write
\begin{equation*}
\varphi_{\cK}(\bsx) = \sum_{\bsk \in \cK} \left[\left(\prod_{j=1}^s c_{k_j}^{(j)}\right)^{-1} \pi(\bsy_{\bsk})\right]  \left(\prod_{j=1}^s c_{k_j}^{(j)}\right) h_{\bsk}(\bsx),
\end{equation*}
where now $(\prod_{j=1}^s c_{k_j}^{(j)}) h_{\bsk}(\bsx)$ is a probability density, and put $$c_{\bsk}:=\left(\prod_{j=1}^s c_{k_j}^{(j)}\right)^{-1} \pi(\bsy_{\bsk}).$$
This presentation of $\varphi_{\cK}$ is exactly of the form \eqref{tensform} with $\cN=\cK$.
Then 
\begin{equation}\label{def:ci2}
c = \sum_{\bsk \in \cK} c_{\bsk} =\int_{[\bsa, \bsb]} \varphi_{\cK}(\bsy) \rd \bsy  \approx \int_{[\bsa, \bsb]} \pi(\bsy) \rd \bsy.
\end{equation}

Now we use the algorithm defined in \eqref{alg_A} with the present $c_{\bsk}$ and $\cN=\cK$, i.e., 
\begin{eqnarray}\label{def:algthm3}
\mathcal{A}_{N,s}(f) & = & \sum_{\bsk \in \cK} \left(\prod_{j=1}^s c_{k_j}^{(j)}\right)^{-1} \frac{\pi(\bsy_{\bsk})}{N_{\bsk}} \sum_{n=0}^{N_{\bsk}-1}f(\Phi_{\bsk}^{-1}(\bsy_n)).
\end{eqnarray}

\begin{theorem}\label{thm_adaptive_hats}
Let $1\le p,q\le \infty$ be such that $1/p+1/q=1$. Assume that the target density  $\overline{\pi}:[\bsa,\bsb] \rightarrow \RR_0^+$ is bounded and satisfies the local H\"older condition \eqref{Hcond2no} with local H\"older constants $L_{\overline\pi}(\bsk)$ and exponent $\beta \in (0,1]$. Let $f:[\bsa,\bsb]\rightarrow \RR$ satisfy  $\|f\|_{p,s,\bsgamma}< \infty$, where the norm is defined by \eqref{def:norm}. Let $\delta>0$ and let $\mathcal{A}_{N,s}$ be the algorithm from \eqref{def:algthm3} based on a digital $(t,s)$-sequence over $\FF_2$ with non-singular upper triangular generating matrices $C_1,\ldots,C_s$. Then we have
\begin{align*}
\left|  \int_{[\bsa, \bsb]} f(\bsx) \overline{\pi}(\bsx) \rd \bsx - \mathcal{A}_{N,s}(f) \right|  \le & 2 \|f\|_{L_\infty} \sum_{\bsk \in \cK^\ast} \prod_{j=1}^s (y_{k_j+1} - y_{k_j}) L_{\overline{\pi}}(\bsk) \, {\rm diam}(R_{\bsk})^{\beta} \\
& +\|f\|_{p,s,\bsgamma}  \frac{\delta+3 r-2}{N} \left[1 + 5 \cdot 2^t \log(2 N)  G_{\bsgamma,q,\bsa,\bsb}(N)\right],
\end{align*}
where $r=r(\delta,N)$ is given by \eqref{def:r} and $G_{\bsgamma,q,\bsa,\bsb}(N)$ is defined in \eqref{def:Gab}.
\end{theorem}

The proof of the theorem works very similarly to the proof of Theorem~\ref{thm_hat_approx_QMC}. We omit the details here.

\section{Quasi-Monte Carlo sampling of hat function approximations of target densities defined on general domains via a partition of unity}\label{sec5:partition}

In this section, we consider the approximation of integrals of the form $$\int_D f(\bsx) \overline{\pi}(\bsx) \rd \bsx,$$ where the probability density function $\overline{\pi}$ is only known up to an unknown normalizing constant, i.e., we are given $\pi$ such that $\overline{\pi} = \pi / z$, for some unknown positive real number~$z$.

\subsection{Partition of unity approximation of a general density}
For target densities that have localized features, e.g., multi-modality, it can be more efficient to build multiple local hat function approximations to adapt to the local features rather than only a global one. We consider the use of a partition-of-unity method to achieve this. Let $\overline\pi$ be a Lebesgue measurable target probability density defined on a Lebesgue measurable domain $D \subseteq \mathbb{R}^s$ with bounded $L^p$ norm
\begin{equation*}
\left(\int_D |\overline\pi(\bsx)|^p \rd \bsx \right)^{1/p} \quad \mbox{for some } p \in [1,\infty].
\end{equation*}
Let $I \in \NN$. Given a set of non-negative functions $\psi_1, \ldots, \psi_I: D \to \mathbb{R}_0^+$ such that $\int_D \psi_i(\bsx) \rd \bsx = 1 $ and a set of positive weights $\alpha_1, \ldots, \alpha_I > 0$ such that $\sum_{i=1}^I \alpha_i = 1$, we can define the function
\begin{equation}
    \Psi(\bsx) = \sum_{i=1}^I \alpha_i \psi_i(\bsx) \quad \mbox{for $\bsx \in D$.} \label{eq:mixture_density}
\end{equation} 
This way, the set of functions $\{\frac{\alpha_i\psi_i}{\Psi}\}_{i = 1}^I$ defines a partition of unity. Assuming that $\Psi(\bsx) > 0$ for all $\bsx \in D$, the target density can be equivalently expressed as
\begin{equation*}
\overline\pi = \overline\pi \, \frac{\sum_{i=1}^I \alpha_i \psi_i}{\Psi} = \sum_{i=1}^I \alpha_i \, \frac{\overline\pi \psi_i}{\Psi} = \frac{1}{z}\sum_{i=1}^I \alpha_i \, \frac{\pi \psi_i}{\Psi}.
\end{equation*}
With a suitable choice of the set of functions $\{\psi_i\}_{i = 1}^I$ and weights $\{\alpha_i\}_{i = 1}^I$, approximating the function $ \overline{\pi} \psi_i / \Psi$ can be an easier task compared to directly approximating $\overline{\pi}$. %
For example, if $ \Psi$ is a good approximation to $\pi$, each of the localized functions will be a perturbation of the localized function $\psi_i$.

In this work, we construct such a partition of unity using the Gaussian mixture, where the functions $\psi_i$ are Gaussian PDFs with different mean vectors $\bmu_i$ and covariance matrices $\bSigma_i$. %
The expectation-maximization algorithm \cite{dempster1977maximum,wu1983convergence} can be used to iteratively identify the weights $\alpha_i$, the mean vectors $\bmu_i$ and the covariance matrices $\bSigma_i$. 

The goal is now to approximate the functions $\widetilde{\pi}^{(i)} := \pi \psi_i/\Psi$ for $i \in \{1,\ldots,I\}$. Suppose each of the covariance matrices of $\psi_i$ has an eigendecomposition $\bSigma_i = \bU_i^{} \bLambda_i^{} \bU_i^\top$, where $\bU_i$ is a unitary matrix consisting of all eigenvectors and $\bLambda_i$ is a diagonal matrix consisting of all corresponding eigenvalues of $\bSigma_i$. Since the tails of $\widetilde{\pi}^{(i)}$ can be controlled by the Gaussian density $\psi_i$, we can introduce the change of variable
\[
\bsz = T_i(\bsx) := \bU_i^\top (\bsx - \bmu_i),
\]
so that $\psi_i \circ T_i^{-1}$, i.e., $\psi_i \circ T_i^{-1}(\bsz)= \psi_i(T_i^{-1}(\bsz))$,
is the density of a zero mean Gaussian random vector with covariance matrix $\bLambda_i$. This allows us to truncate the domain of the function $\widetilde{\pi}^{(i)}$ to a rotated interval 
\begin{equation}\label{rot:int}
B_i = \big\{\bsx = T_i^{-1}(\bsz) \ : \ \bsz \in [-\bsa^{(i)}, \bsa^{(i)}] \big\} %
\end{equation}
for some $\bsa^{(i)} > 0$ (to be understood component-wise). Applying the change of variable, the approximation of $\widetilde{\pi}^{(i)}(\bsx)$ for $\bsx \in B_i$ can be obtained by approximating the function $\widetilde{\pi}^{(i)} \circ T_i^{-1}$ using the coordinate-wise adaptive approach on $\bsz \in [-\bsa^{(i)}, \bsa^{(i)}]$.

The approximation of the target density satisfies
\begin{eqnarray*}
\left| \overline\pi(\bsx) - \frac{1}{z}\sum_{i=1}^I \alpha_i \widetilde{\pi}^{(i)}(\bsx) \mathbb{1}_{B_i}(\bsx) \right| & \le &  \frac{1}{z}\sum_{i=1}^I \alpha_i \left| \widetilde{\pi}^{(i)}(\bsx) -  \widetilde{\pi}^{(i)}(\bsx) \mathbb{1}_{B_i}(\bsx) \right| \\  
& = &  \frac{\overline\pi(\bsx)}{\Psi(\bsx)} \sum_{i = 1}^I \alpha_i \left|\psi^{(i)}(\bsx) - \psi^{(i)}(\bsx) \mathbb{1}_{B_i}(\bsx) \right| \\
& \leq & \frac{\overline\pi(\bsx)}{\Psi(\bsx)} \sum_{i = 1}^I \alpha_i \sup_{\bsx \notin B_i} \psi_i(\bsx),
\end{eqnarray*}
where $\mathbb{1}_{B_i}$ denotes the indicator function of the rotated interval $B_i$ for $i \in \{1,\ldots,I\}$. Collecting the eigenvalues of $\bSigma_i$ into a vector $\blambda_i$, we can choose $\bsa^{(i)} = a \surd \blambda_i$ for some $a>0$, and hence $B_i$, such that $\psi_i(\bsx) \le \varepsilon$ for all $\bsx \notin B_i$ for some $\varepsilon>0$ chosen to be sufficiently small. Therefore for all $\bsx \in D$ we have
\begin{equation*}
\left| \overline\pi(\bsx) - \frac{1}{z}\sum_{i=1}^I \alpha_i \widetilde{\pi}^{(i)}(\bsx) \mathbb{1}_{B_i}(\bsx) \right|  \le \varepsilon  \frac{\overline\pi(\bsx)}{\Psi(\bsx)}.
\end{equation*}

\subsection{Adaptive hat function approximation on rotated intervals of partition of unity functions}

The computation of $\int_D f(\bsx) \overline\pi(\bsx) \rd \bsx$ consists of two integration problems, namely, the approximation of $\int_D f(\bsx) \pi(\bsx) \rd \bsx$ and the approximation of $z = \int_D \pi(\bsx) \rd \bsx$. We first consider the approximation of the first integral. Using the partition of unity 
\begin{eqnarray*}
\int_D f(\bsx) \sum_{i=1}^I \alpha_i   \widetilde{\pi}^{(i)}(\bsx) \mathbb{1}_{B_i}(\bsx) \rd \bsx & = & \sum_{i=1}^I \alpha_i \int_{B_i} f(\bsx) \widetilde{\pi}^{(i)}(\bsx) \rd \bsx \\ 
& = & \sum_{i=1}^I \alpha_i \int_{[-\bsa^{(i)}, \bsa^{(i)}]} f(T_i^{-1} (\bsz)) \widetilde{\pi}^{(i)}(T_i^{-1} (\bsz)) \rd \bsz, %
\end{eqnarray*}
in which the key is the approximation of each $\widetilde{\pi}^{(i)} \circ T_i^{-1}$. Then the constant $z$ can be computed using the approximations of $\widetilde{\pi}^{(i)} \circ T_i^{-1}$ for $i \in \{1, \ldots, I\}$. To approximate $\widetilde{\pi}^{(i)} \circ T_i^{-1}$, we need to assume that this function satisfies a H\"older condition for some exponent $\beta \in (0,1]$, i.e., for every pair $\bsz, \bsz' \in [-\bsa^{(i)}, \bsa^{(i)}]$ we have
\begin{align*}
|\widetilde{\pi}^{(i)}( T_i^{-1}(\bsz)) - \widetilde{\pi}^{(i)}(T_i^{-1}(\bsz')) | 
& \le C_{\widetilde{\pi}^{(i)}} \,\|\bU_i ( \bsz - \bsz') \|_2^{\beta}  = C_{\widetilde{\pi}^{(i)}} \,\|( \bsz - \bsz') \|_2^{\beta}
\end{align*}
for some positive $ C_{\widetilde{\pi}^{(i)}}$ depending on $\widetilde{\pi}^{(i)}$. In this scenario, the use of the $\ell_2$-norm instead of the $\ell_\infty$-norm used in previous sections is beneficial, as it enables us to utilize unitary transformations.

\begin{proposition}\label{mix_holder}
We assume the unnormalized target density $\pi$ and all of the functions $\psi_i$, $i \in \{1, \ldots, I\}$, are H\"{o}lder continuous with some exponent $\beta \in (0,1]$. In addition, we assume the unnormalized target density $\pi$ and the function $\Psi$ satisfies 
\[
 C_{\pi, \Psi}:= \sup_{\bsx \in D} \frac{\pi(\bsx)}{\Psi(\bsx)} < \infty.
\]
Then each of the functions $\widetilde{\pi}^{(i)}$, $i \in \{1, \ldots, I\}$, satisfies a H\"{o}lder condition with the same exponent $\beta \in (0,1]$. That is, there exists a constant $L_{\widetilde{\pi}^{(i)}}$ such that
\begin{equation*}
|\widetilde{\pi}^{(i)}( \bsx) - \widetilde{\pi}^{(i)}(\bsy) | \le L_{\widetilde{\pi}^{(i)}} \| \bsx -  \bsy \|^\beta_2 \quad \mbox{ for all $\bsx, \bsy \in B_i$.}
\end{equation*}
\end{proposition}

\begin{proof}
By the definition of the partition of unity in \eqref{eq:mixture_density}, we have the property
\[
\sup_{\bsx \in D} 
\frac{\psi_i(\bsx)}{\Psi(\bsx)} \le \frac{1}{\alpha_i}.
\]
In addition, the  H\"{o}lder continuity assumption on $\psi_i$ also makes $\Psi$  H\"{o}lder continuous. Thus, we have
\begin{eqnarray*}
|\widetilde{\pi}^{(i)}(\bsx) - \widetilde{\pi}^{(i)}(\bsy) | & \le & \left| \pi(\bsx) \frac{\psi_i(\bsx)}{\Psi(\bsx)} - \pi(\bsy) \frac{\psi_i(\bsx)}{\Psi(\bsx)} \right| + \left|\pi(\bsy) \frac{\psi_i(\bsx)}{\Psi(\bsx)} - \pi(\bsy) \frac{\psi_i(\bsy)}{\Psi(\bsy)} \right| \\ & \le & \frac{\psi_i(\bsx)}{\Psi(\bsx)}  |\pi(\bsx) - \pi(\bsy) |  + \pi(\bsy) \left|\frac{\psi_i(\bsx)}{\Psi(\bsx)} - \frac{\psi_i(\bsy)}{\Psi(\bsy)} \right| \\ & \le &  \frac{L_{\pi}}{\alpha_i} \|\bsx-\bsy\|_2^\beta + \pi(\bsy) \left|\frac{\psi_i(\bsx)}{\Psi(\bsx)} - \frac{\psi_i(\bsy)}{\Psi(\bsy)} \right| ,
\end{eqnarray*}
where $L_\pi$ is the H\"older constant for the unnormalized density $\pi$ with respect to the $\ell_2$-norm. For the second term in the above upper bound, we have
\begin{align*}
\pi(\bsy) \left|\frac{\psi_i(\bsx)}{\Psi(\bsx)} - \frac{\psi_i(\bsy)}{\Psi(\bsy)} \right| & = \frac{\pi(\bsy)}{\Psi(\bsx) \Psi(\bsy)}  \left|\psi_i(\bsx)\Psi(\bsy) - \psi_i(\bsx)\Psi(\bsx)+  \psi_i(\bsx)\Psi(\bsx)- \psi_i(\bsy)\Psi(\bsx) \right|  \\
& \leq \frac{\pi(\bsy)}{\Psi(\bsx) \Psi(\bsy)} \left( \psi_i(\bsx) \left|\Psi(\bsy) - \Psi(\bsx) \right|+  \Psi(\bsx)\left|\psi_i(\bsx)- \psi_i(\bsy) \right|\right) \\
& =  \frac{\pi(\bsy)}{\Psi(\bsy)}\left(  \frac{\psi_i(\bsx)}{\Psi(\bsx)} \left|\Psi(\bsy) - \Psi(\bsx) \right|+   \left|\psi_i(\bsx)- \psi_i(\bsy) \right|\right) \\
& \leq C_{\pi, \Psi} \left(  \frac{1}{\alpha_i} \left|\Psi(\bsy) - \Psi(\bsx) \right|+   \left|\psi_i(\bsx)- \psi_i(\bsy) \right|\right).
\end{align*}
This way, as long as each $\psi_i$ is H\"{o}lder continuous (with exponent $\beta \in (0,1]$), then $\widetilde{\pi}^{(i)}$ satisfies a H\"older condition with the same exponent $\beta \in (0,1]$.
\end{proof}

\subsection{The combined algorithm}\label{sec:5.3}

For every $i \in \{1,\ldots,I\}$, we apply the adaptive hat function approach to approximate the function $\widetilde{\pi}^{(i)} \circ T_i^{-1}$ on $[-\bsa^{(i)},\bsa^{(i)}]$. For each rotated interval $B_i$ (see \eqref{rot:int}), this defines grid points on the transformed coordinates $\bsz$ as $y^{(i,j)}_k$, $k \in \{0,1,\ldots,K^{(i)}_j\}$, for the coordinate $z_j$, $j \in [s]$. This way, we have the index set $\mathcal{K}_i := \bigotimes_{j=1}^s \{0, 1, \ldots, K^{(i)}_{j}\}$ and grid points
\begin{equation}\label{gridpts_mix}
\left\{ \bsy^{(i)}_{\bsk} = ( y^{(i,1)}_{k_1}, y^{(i,2)}_{k_2}, \ldots, y^{(i,s)}_{k_s} ) \ : \ \bsk = (k_1, \ldots, k_s) \in \mathcal{K}_i \right\}
\end{equation}
according to \eqref{gridpts_general}. We approximate $\widetilde{\pi}^{(i)} \circ T_i^{-1}$ by
\begin{equation}\label{eq:phi_mix}
\widetilde{\varphi}^{(i)}_{\cK_i}(\bsz) = \sum_{\bsk \in \cK_i} c_{\bsk}^{(i)}  \varphi_{\bsk,i}(\bsz) \quad \mbox{ for $\bsz \in [-\bsa^{(i)},\bsa^{(i)}]$,}
\end{equation}
where 
\begin{equation}
    c_{\bsk}^{(i)}:=\left(\prod_{j=1}^s c_{k_j}^{(j)}\right)^{-1} \widetilde{\pi}^{(i)}(T_i^{-1}(\bsy_{\bsk}^{(i)}))\quad \text{and} \quad \varphi_{\bsk,i}(\bsz) := \left(\prod_{j=1}^s c_{k_j}^{(j)}\right) h_{\bsk,i}(\bsz).
\end{equation}
Here $\varphi_{\bsk,i}(\bsz)$ is the hat function $h_{\bsk,i}$ normalized to a probability density function and the normalizing constants are $c_{k_j}^{(j)}$ (see Appendix~\ref{sec:aux}, Eq. \eqref{norm:const} and \eqref{eq_phikj}). 

We have $\widetilde{\pi}^{(i)}(T_i^{-1}(\bsz)) \approx \widetilde{\varphi}^{(i)}_{\cK_i}(\bsz)$, and thus
\begin{equation}\label{def:ci}
c^{(i)} = \sum_{\bsk \in \cK_i} c_{\bsk}^{(i)} \approx \int_{B_i} \widetilde{\pi}^{(i)}(\bsx) \,\mathrm{d} \bsx.
\end{equation}

\begin{lemma}\label{le:5}
We assume the unnormalized target density $\pi$ and all of the functions $\psi_i$, $i \in \{1, \ldots, I\}$, are H\"{o}lder continuous with exponent $\beta \in (0,1]$. In addition, we assume the unnormalized target density $\pi$ and the function $\Psi$ satisfies 
\[
\sup_{\bsx \in D} \frac{\pi(\bsx)}{\Psi(\bsx)} = C_{\pi, \Psi} < \infty.
\]
Define the normalized approximation of $\overline\pi(\bsx)$ by
\begin{equation}\label{eq:mix_approx}
\overline\varphi(\bsx) = \frac1c \sum_{i = 1}^I \alpha_i \ \widetilde{\varphi}^{(i)}_{\cK_i}(T_i(\bsx)) \quad \mbox{ for }\ \bsx \in \bigcup_{i = 1}^{I} B_i, 
\end{equation}
where $$c = \sum_{i = 1}^I \alpha_i c^{(i)},$$ and where $\widetilde{\varphi}^{(i)}_{\cK_i}$ is defined in \eqref{eq:phi_mix}. Then we have
\begin{eqnarray*}
\lefteqn{\int_D |\overline\pi(\bsx) - \overline\varphi(\bsx) | \rd \bsx}\\
& \leq & 2 \varepsilon \int_D \frac{\overline\pi(\bsx)}{\Psi(\bsx)} \rd \bsx + \frac2z \sum_{i = 1}^I \alpha_i  \sum_{\bsk \in \cK_i^\ast} \prod_{j=1}^s (y_{k_j+1}^{(i,j)} - y_{k_j}^{(i,j)}) L_{\widetilde\pi^{(i)}}(\bsk) \, {\rm diam}(R^{(i)}_{\bsk})^{\beta},
\end{eqnarray*}
where $\cK^\ast_i := \bigotimes_{j=1}^s \{0, 1,\ldots, K^{(i)}_{j}-1\}$, where $R^{(i)}_{\bsk} := \bigotimes_{j=1}^s [y^{(i,j)}_{k_j},y^{(i,j)}_{k_j+1})$, corresponding to the grid points \eqref{gridpts_mix}, $z$ is the normalizing constant of the target density $\pi$, and $\varepsilon > 0$ is defined such that $\psi_i(\bsx) \le \varepsilon$ for all $\bsx \notin B_i$ and $i \in \{1,\ldots,I\}$.
\end{lemma}

\begin{proof}
Using a proof similar to that of Lemma~\ref{lem_pi_approx}, we have
\[
\int_D |\overline\pi(\bsx) - \overline\varphi(\bsx) | \rd \bsx \leq \frac2z \int_D |\pi(\bsx) - \varphi(\bsx) | \rd \bsx ,
\]
where 
\[
\varphi(\bsx) = \sum_{i = 1}^I \alpha_i \ \widetilde{\varphi}^{(i)}_{\cK_i}(T_i(\bsx)) \quad \mbox{ for }\ \bsx \in \bigcup_{i = 1}^{I} B_i
\]
is the unnormalized version of $\overline\varphi(\bsx)$ and where the rotated intervals $B_i$ are defined in \eqref{rot:int}. Then, we consider the following bound on the unnormalized densities,
\begin{align*}
& \hspace{-12pt} \int_D |\pi(\bsx) - \varphi(\bsx) | \rd \bsx \\
& \leq \sum_{i = 1}^I \alpha_i \int_D |\widetilde\pi^{(i)}(\bsx) - \widetilde{\varphi}^{(i)}_{\cK_i}(T_i(\bsx)) | \rd \bsx \\
& \leq \sum_{i = 1}^I \alpha_i \left(\int_D |\widetilde\pi^{(i)}(\bsx) - \widetilde\pi^{(i)}(\bsx)\mathbb{1}_{B_i}(\bsx)  | \rd \bsx + \int_D |\widetilde\pi^{(i)}(\bsx)\mathbb{1}_{B_i}(\bsx)  - \widetilde{\varphi}^{(i)}_{\cK_i}(T_i(\bsx)) | \rd \bsx\right) \\
& = \sum_{i = 1}^I \alpha_i \left(\int_D \widetilde\pi^{(i)}(\bsx)|1 - \mathbb{1}_{B_i}(\bsx)  | \rd \bsx + \int_{[-\bsa^{(i)},\bsa^{(i)}]} |\widetilde\pi^{(i)}(T_i^{-1}(\bsz))  - \widetilde{\varphi}^{(i)}_{\cK_i}(\bsz) | \rd \bsz\right) \\
& = \sum_{i = 1}^I \alpha_i \int_D \frac{\pi(\bsx) \psi_i(\bsx)}{\Psi(\bsx)}|1 - \mathbb{1}_{B_i}(\bsx)  | \rd \bsx + \sum_{i = 1}^I \alpha_i  \int_{[-\bsa^{(i)},\bsa^{(i)}]} |\widetilde\pi^{(i)}(T_i^{-1}(\bsz))  - \widetilde{\varphi}^{(i)}_{\cK_i}(\bsz) | \rd \bsz.
\end{align*}
We estimate the two main parts separately. 

For the first term recall that $\psi_i(\bsx) \leq \varepsilon$ for all $\bsx\not\in B_i$ and $\sum_{i = 1}^I \alpha_i =1$. Then we have
\[
\sum_{i = 1}^I \alpha_i \int_D \frac{\pi(\bsx) \psi_i(\bsx)}{\Psi(\bsx)}|1 - \mathbb{1}_{B_i}(\bsx)  | \rd \bsx \leq \varepsilon \sum_{i = 1}^I \alpha_i \int_D \frac{\pi(\bsx)}{\Psi(\bsx)} \rd \bsx = \varepsilon \int_D \frac{\pi(\bsx)}{\Psi(\bsx)} \rd \bsx.
\]

In order to estimate the second part we can apply Lemma~\ref{le_adapt_pi_holder} for each $\widetilde{\varphi}^{(i)}_{\cK_i}(\bsz)$. Indeed, in Proposition \ref{mix_holder}, we established that $\widetilde\pi^{(i)}$ satisfies a H\"older condition of the form
\begin{equation*}%
|\widetilde\pi^{(i)}(\bsx)-\widetilde\pi^{(i)}(\bsy)| \le L_{\widetilde\pi^{(i)}} \, \|\bsx-\bsy\|_2^{\beta} \quad \mbox{ for all $\bsx, \bsy \in B_i$,}
\end{equation*}
or equivalently
\begin{eqnarray*}
|\widetilde\pi^{(i)}(T_i^{-1}(\bsz))-\widetilde\pi^{(i)}(T_i^{-1}(\bsz'))| & \le & L_{\widetilde\pi^{(i)}} \, \|T_i^{-1}(\bsz)-T_i^{-1}(\bsz')\|_2^{\beta}\nonumber\\
& = & L_{\widetilde\pi^{(i)}} \, \|\bsz-\bsz'\|_2^{\beta} \quad \mbox{ for all $\bsz, \bsz' \in [-\bsa^{(i)},\bsa^{(i)}]$.}
\end{eqnarray*}
Now we consider a local version
\begin{equation*}
|\widetilde\pi^{(i)}(T_i^{-1}(\bsz))-\widetilde\pi^{(i)}(T_i^{-1}(\bsz'))| \le L_{\widetilde\pi^{(i)}}(\bsk) \, \|\bsz-\bsz'\|_2^{\beta} \quad \mbox{ for all $\bsz, \bsz' \in R^{(i)}_{\bsk}$,}
\end{equation*}
for all $\bsk\in \cK^\ast_i$. Then, for $\bsz \in R^{(i)}_{\bsk}$ we have
$$|\widetilde\pi^{(i)}(T_i^{-1}(\bsz)) - \widetilde\varphi^{(i)}_{\cK_i}(\bsz) | \le L_{\widetilde\pi^{(i)}}(\bsk) \, {\rm diam}(R^{(i)}_{\bsk})^{\beta}.$$
This leads to 
\[
\int_{[-\bsa^{(i)},\bsa^{(i)}]} |\widetilde\pi^{(i)}(T_i^{-1}(\bsz))  - \widetilde{\varphi}^{(i)}_{\cK_i}(\bsz) | \rd \bsz \leq 
\sum_{\bsk \in \cK_i^\ast} \prod_{j=1}^s (y_{k_j+1}^{(j)} - y_{k_j}^{(j)}) L_{\widetilde\pi^{(i)}}(\bsk) \, {\rm diam}(R^{(i)}_{\bsk})^{\beta}.
\]
In summary, we have the bound
\begin{align*}
\lefteqn{\int_D |\overline\pi(\bsx) - \overline\varphi(\bsx) | \rd \bsx}\\
& \le \frac{2}{z}\left(\varepsilon \int_D \frac{\pi(\bsx)}{\Psi(\bsx)} \rd \bsx + \sum_{i = 1}^I \alpha_i \sum_{\bsk \in \cK_i^\ast} \prod_{j=1}^s (y_{k_j+1}^{(j)} - y_{k_j}^{(j)}) L_{\widetilde\pi^{(i)}}(\bsk) \, {\rm diam}(R^{(i)}_{\bsk})^{\beta}  \right)\\
& = 2 \varepsilon \int_D \frac{\overline\pi(\bsx)}{\Psi(\bsx)} \rd \bsx + \frac2z \sum_{i = 1}^I \alpha_i  \sum_{\bsk \in \cK_i^\ast} \prod_{j=1}^s (y_{k_j+1}^{(j)} - y_{k_j}^{(j)}) L_{\widetilde\pi^{(i)}}(\bsk) \, {\rm diam}(R^{(i)}_{\bsk})^{\beta},
\end{align*}
which concludes the proof.
\end{proof}

Using each component of the approximate density \eqref{eq:mix_approx}, now we define for every $i \in \{1,\ldots,I\}$ the algorithm 
\begin{align}\label{alg:i}
\mathcal{A}^{(i)}_{\cK_i}(f \circ T_i^{-1}) = &  \sum_{\bsk \in \cK_i} \frac{c_{\bsk}^{(i)} }{N_{\bsk}^{(i)}} \sum_{n=0}^{N_{\bsk}^{(i)}-1} f(T_i^{-1}(\Phi^{-1}_{\bsk,i}(\bsy_n))) \nonumber \\ \approx & \sum_{\bsk \in \cK_i} c_{\bsk}^{(i)}  \int_{B_i} f(\bsx) \varphi_{\bsk, i}(T_i(\bsx)) \rd \bsx, \nonumber %
\end{align}
where $\Phi^{-1}_{\bsk, i}$ is the inverse CDF of the normalized hat function $h_{\bsk,i}$, $(\bsy_n)_{n \ge 0}$ is a QMC sequence in the unit cube $[0,1]^s$ and $N_{\bsk}^{(i)} \in \mathbb{N}$ for $\bsk \in \mathcal{K}_i$. Combining all components of \eqref{eq:mix_approx} together, the final combined approximation algorithm is now of the form
\begin{eqnarray}\label{def:alg2}
\mathcal{A}_{\bscK}(f) & = & \frac1c \sum_{i=1}^I \alpha_i \mathcal{A}^{(i)}_{\cK_i}(f \circ T_i^{-1}) \nonumber\\
& = & \frac1c  \sum_{i=1}^I \alpha_i \sum_{\bsk \in \cK_i}  \frac{c_{\bsk}^{(i)}}{N_{\bsk}^{(i)}} \sum_{n=0}^{N_{\bsk}^{(i)}-1} f(T_i^{-1}(\Phi^{-1}_{\bsk,i}(\bsy_n))),
\end{eqnarray}
where \(c = \sum_{i = 1}^I \alpha_i c^{(i)}\), $\bscK = (\cK_1, \cK_2, \ldots, \cK_I)$, and $\sum_{i = 1}^I \alpha_i = 1$.

To apply the combined algorithm, we proceed in the following way. We approximate the target density by choosing $\alpha_i$ and $\psi_i$ for $i \in \{1,\ldots,I\}$. Then we apply the transformation $T_i$ and choose a bounded interval $[-\bsa^{(i)}, \bsa^{(i)}]$ with $\bsa^{(i)} = a \surd \blambda_i$ for some $a>0$, where $\blambda_i$ is the vector of eigenvalues of $\bSigma_i$, in the transformed coordinates such that the function $\psi_i$ is bounded by $\varepsilon$ outside the rotated interval $B_i$ (see \eqref{rot:int}) in the original coordinates for some $\varepsilon>0$ that is chosen to be sufficiently small.  We use an adaptive hat function approximation on $[-\bsa^{(i)}, \bsa^{(i)}]$ to approximate $\widetilde\pi_i\circ T_i^{-1}$. In the numerical examples in Section~\ref{sec:numerical} we choose the grid points using an adaptive refinement strategy, see the last lines of Section~\ref{sec:3.2} for more details. %
The numbers $(N_{\bsk}^{(i)})_{\bsk \in \cK_i}$ are chosen according to the weight given to the corresponding hat function (normalized to a PDF). 

The following theorem establishes a bound on the integration error of the algorithm \eqref{def:alg2}, which holds for arbitrary grids. The three parts of the error bound reflect the steps of the procedure of applying the combined algorithm, explained in the previous paragraph.

\begin{theorem}
Let the unnormalized target density  $\pi:D \rightarrow \RR_0^+$ and each of the functions $\psi_i$, $i \in \{1, \ldots, I\}$ satisfy all assumptions of Proposition~\ref{mix_holder}. Given the general weights $\bsgamma=\{\gamma_{\uu}\}_{\uu \subseteq [s]}$ and the corresponding order-dependent weights
\begin{equation}\label{ord:dep:w}
\bsGamma := (\Gamma_v)_{v=1}^s, \quad \Gamma_v := \Big( \sum_{\stackrel{\uu \subseteq [s]}{|\uu| = v}} \frac{1}{\gamma_{\uu}^p} \Big)^{-1/p},
\end{equation} 
we assume that $f:D \rightarrow \RR$ has a bounded $\bsGamma$-weighted $p$-norm with $p \in [1,\infty]$ in the form of 
\begin{equation}\label{def:norm17}
\|f\|_{p,s,\bsGamma} := \left(  \sum_{v=0}^s \left( \frac{1}{\Gamma_v} \sup_{\bsx \in D} \sum_{j_1, j_2, \ldots, j_{v}=1}^s \left| \frac{\partial^{v} f}{\partial x_{j_1} \ldots \partial x_{j_{v}}}(\bsx ) \right| \right)^p  \right)^{1/p} < \infty.
\end{equation}
Let $\varepsilon >0$. For $i \in \{1,\ldots,I\}$ assume that $\bsa^{(i)}$ is such that $\psi_i$ is bounded by $\varepsilon$ outside the rotated interval $B_i=T_i^{-1}([-\bsa^{(i)},\bsa^{(i)}])$ and use an adaptive hat function approximation on $[-\bsa^{(i)},\bsa^{(i)}]$ to approximate $\widetilde\pi_i\circ T_i^{-1}$ with index set $\mathcal{K}_i := \bigotimes_{j=1}^s \{0, 1, \ldots, K^{(i)}_{j}\}$ and grid points
\begin{equation*}
\left\{ \bsy^{(i)}_{\bsk} = ( y^{(i,1)}_{k_1}, y^{(i,2)}_{k_2}, \ldots, y^{(i,s)}_{k_s} ) \ : \ \bsk = (k_1, \ldots, k_s) \in \mathcal{K}_i \right\}.
\end{equation*} 
Furthermore, choose numbers $(N_{\bsk}^{(i)})_{\bsk \in \cK_i}$. Then for the algorithm $\mathcal{A}_{\bscK}$ in \eqref{def:alg2} based on a digital $(t,s)$-sequence $(\bsy_n)_{n \ge 0}$ over $\FF_2$ with non-singular upper triangular generating matrices we have
\begin{eqnarray*}
\lefteqn{\left|\int_D f(\bsx) \overline\pi(\bsx) \rd \bsx - \mathcal{A}_{\bscK}(f) \right|}\\
& \le &  2 \|f\|_{L_\infty} \left( \varepsilon \int_D \frac{\overline\pi(\bsx)}{\Psi(\bsx)} \rd \bsx + \frac1z \sum_{i = 1}^I \alpha_i  \sum_{\bsk \in \cK_i^\ast} \prod_{j=1}^s (y_{k_j+1}^{(j)} - y_{k_j}^{(j)}) L_{\widetilde\pi^{(i)}}(\bsk) \, {\rm diam}(R^{(i)}_{\bsk})^{\beta} \right) \\
&&+ \frac{\|f\|_{p,s,\bsGamma}}c \sum_{i=1}^I \frac{\alpha_i\, c^{(i)} \, (\delta+3 r^{(i)}-2)}{N^{(i)}} \left[1 + 5 \cdot 2^t \log(2 N^{(i)})\, G_{\bsgamma,q,-\bsa^{(i)},\bsa^{(i)}}(N^{(i)})\right].
\end{eqnarray*}
Here, for $i \in \{1,\ldots,I\}$, $c^{(i)}$ is defined in \eqref{def:ci}, $r^{(i)}=r^{(i)}(\delta,N^{(i)})$ is the quantity from \eqref{def:r}, $N^{(i)} = \sum_{\bsk\in\cK_i} N_{\bsk}^{(i)}$, and $G_{\bsgamma,q,-\bsa^{(i)}\bsa^{(i)}}$ is defined in \eqref{def:Gab} with $q \ge 1$ such that $1/p+1/q=1$, $z$ is the normalizing constant of the target density $\pi$, and $c = \sum_{i=1}^I \alpha_i c^{(i)}$.
\end{theorem}

\begin{proof}
We have 
\begin{align}\label{eq1_thm4}
\lefteqn{\left|\int_D f(\bsx) \overline\pi(\bsx) \rd \bsx - \mathcal{A}_{\bscK}(f) \right|} \nonumber\\
& \leq \left|\int_D f(\bsx) ( \overline\pi(\bsx) - \overline\varphi(\bsx) ) \rd \bsx \right| + \left|\int_D f(\bsx) \overline\varphi(\bsx) \rd \bsx - \mathcal{A}_{\bscK}(f) \right|.
\end{align}

For the first part on the right-hand side, we can use Lemma~\ref{le:5} in order to obtain 
\begin{align*}
\lefteqn{\left|\int_D f(\bsx) ( \overline\pi(\bsx) - \overline\varphi(\bsx) ) \rd \bsx \right|}\\
& \leq \|f\|_{L_\infty} \int_D |\overline\pi(\bsx) - \overline\varphi(\bsx) | \rd \bsx \\
& \leq \|f\|_{L_\infty} \left(2 \varepsilon \int_D \frac{\overline\pi(\bsx)}{\Psi(\bsx)} \rd \bsx + \frac2z \sum_{i = 1}^I \alpha_i  \sum_{\bsk \in \cK_i^\ast} \prod_{j=1}^s (y_{k_j+1}^{(j)} - y_{k_j}^{(j)}) L_{\widetilde\pi^{(i)}}(\bsk) \, {\rm diam}(R^{(i)}_{\bsk})^{\beta} \right).
\end{align*}

Now we estimate the second part. We have
\begin{align}\label{bound_gen_case}
\lefteqn{\left|\int_D f(\bsx) \overline\varphi(\bsx) \rd \bsx - \mathcal{A}_{\bscK}(f) \right|} \nonumber \\
& = \left|\frac{1}{c}\sum_{i=1}^I \alpha_i\int_D f(\bsx) \widetilde{\varphi}^{(i)}_{\cK_i}(T_i(\bsx)) \rd \bsx - \frac{1}{c}\sum_{i=1}^I \alpha_i \mathcal{A}_{\cK_i}^{(i)}(f \circ T_i^{-1}) \right| \nonumber \\
& =  \left|\frac{1}{c}\sum_{i=1}^I \alpha_i\int_{[-\bsa^{(i)},\bsa^{(i)}]} f(T_i^{-1}(\bsz)) \widetilde{\varphi}^{(i)}_{\cK_i}(\bsz) \rd \bsz - \frac{1}{c}\sum_{i=1}^I \alpha_i \mathcal{A}_{\cK_i}^{(i)}(f \circ T_i^{-1}) \right| \nonumber  \\ 
& \le \frac{1}{c} \sum_{i=1}^I \alpha_i \left|\int_{[-\bsa^{(i)},\bsa^{(i)}]} g_i(\bsz) \widetilde{\varphi}^{(i)}_{\cK_i}(\bsz) \rd \bsz - \mathcal{A}_{\cK_i}^{(i)}(g_i) \right| \nonumber \\
& = \frac{1}{c} \sum_{i=1}^I \alpha_i c^{(i)} \left|\int_{[-\bsa^{(i)},\bsa^{(i)}]} g_i(\bsz) \left(\frac{1}{c^{(i)}}\widetilde{\varphi}^{(i)}_{\cK_i}(\bsz) \right) \rd \bsz - \frac{1}{c^{(i)}} \mathcal{A}_{\cK_i}^{(i)}(g_i) \right|
\end{align}
where $g_i = f\circ T_i^{-1}$. Applying the second half of Theorem~\ref{thm_adaptive_hats}, we have
\begin{align}\label{bd:erraigif}
\lefteqn{\left|\int_{[-\bsa^{(i)},\bsa^{(i)}]} g_i(\bsz) \left( \frac{1}{c^{(i)}} \widetilde{\varphi}^{(i)}_{\cK_i}(\bsz) \right) \rd \bsz - \frac{1}{c^{(i)}}\mathcal{A}_{\cK_i}^{(i)}(g_i) \right|} \nonumber \\
& \leq \|g_i\|_{p,s,\bsgamma}  \frac{\delta+3 r^{(i)}-2}{N^{(i)}} \left[1 + 5 \cdot 2^t \log(2 N^{(i)})  G_{\bsgamma,q,-\bsa^{(i)},\bsa^{(i)}}(N^{(i)})\right], 
\end{align}
where $N^{(i)} = \sum_{\bsk\in\cK_i} N_{\bsk}^{(i)}$ and where 
\begin{equation*}
\|g_i\|_{p,s,\bsgamma}=\left( \sum_{\uu \subseteq [s]} \left( \frac{1}{\gamma_\uu} \sup_{\bsz \in [-\bsa^{(i)}, \bsa^{(i)}]} \left | \frac{\partial^{|\uu|} g_i}{\partial \bsz_\uu} (\bsz)  \right | \right)^p \right)^{1/p}.
\end{equation*}
Now we show how to bound this norm uniformly for $i \in \{1,\ldots,I\}$ by the norm \eqref{def:norm17} of $f$.

For $\uu = \{ u_1, u_2, \ldots, u_{v}\} \subseteq [s]$, where $v=|\uu|$, we have
\begin{align*}
\frac{\partial^{v} g_i}{\partial \bsz_\uu}(\bsz) =  \sum_{j_1, j_2, \ldots, j_{v}=1}^s \frac{\partial^{v} f}{\partial x_{j_1} \dots \partial x_{j_{v}}}( \bsx)\, a^{(i)}_{j_1, u_1} \cdots a^{(i)}_{j_{v}, u_{v}},
\end{align*}
where $\bsx = \bU_i \bsz + \bmu_i$, and $a_{j,u}^{(i)}$ is the element in row $j$ and column $u$ of $\bU_i$. We can bound $|a^{(i)}_{j_1, u_1} \cdots a^{(i)}_{j_{v}, u_{v}}|$ by $1$, because $\bU_i$ is a unitary matrix and hence we have
\[
|a^{(i)}_{j,u}|^2 \le \sum_{k=1}^s |a^{(i)}_{j,k}|^2 =1 \quad \mbox{for all $j,u \in [s]$}.
\]
Therefore, for any subset $\emptyset \neq \uu \subseteq [s]$ with $|\uu| = v$ we have
\begin{align*}
\left| \frac{\partial^{v} g_i}{\partial \bsz_\uu}(\bsz) \right| \le  \sum_{j_1, j_2, \ldots, j_{v}=1}^s \left| \frac{\partial^{v} f}{\partial x_{j_1} \cdots \partial x_{j_{v}}}(T_i^{-1}(\bsz)) \right|.
\end{align*}
For $\uu = \emptyset$ we have $v=0$ and we interpret the above inequality as $|g_i(\bsz)| = |f(T^{-1}_i(\bsz))|$. Thus
\begin{align*}
\|g_i\|_{p,s,\bsgamma}^p = & \sum_{\uu \subseteq [s]} \left(\frac{1}{\gamma_{\uu}} \sup_{\bsz \in [-\bsa^{(i)}, \bsa^{(i)}]} \left|\frac{\partial^{|\uu|} g_i}{\partial \bsz_{\uu}}(\bsz) \right| \right)^p \\ 
\le &  \sum_{\uu \subseteq [s]} \left(\frac{1}{\gamma_\uu} \sup_{\bsx \in D} \sum_{j_1, j_2, \ldots, j_{|\uu|}=1}^s \left| \frac{\partial^{|\uu|} f}{\partial x_{j_1} \ldots \partial x_{j_{|\uu|}}}(\bsx) \right| \right)^p \\ 
= & \sum_{v=0}^s \sum_{\stackrel{\uu \subseteq [s]}{|\uu| = v}} \frac{1}{\gamma_\uu^p} \left( \sup_{\bsx \in D} \sum_{j_1, j_2, \ldots, j_{v}=1}^s \left| \frac{\partial^{v} f}{\partial x_{j_1} \ldots \partial x_{j_{v}}}(\bsx ) \right| \right)^p.
\end{align*}
Using the norm \eqref{def:norm17} now gives $\|g_i\|_{p,s,\bsgamma} \le \|f\|_{p,s, \bsGamma}$. Applying this estimate to \eqref{bd:erraigif} gives 
\begin{align*}
\lefteqn{\left|\int_{[-\bsa^{(i)},\bsa^{(i)}]} g_i(\bsz) \left(\frac{\widetilde{\varphi}^{(i)}_{\cK_i}(\bsz)}{c^{(i)}} \right) \rd \bsz - \frac{1}{c^{(i)}}\mathcal{A}_{N^{(i)},s}(g_i) \right|} \\
& \leq  \|f\|_{p,s,\bsGamma} \frac{\delta+3 r^{(i)}-2}{N^{(i)}} \left[1 + 5 \cdot 2^t \log(2 N^{(i)})  G_{\bsgamma,q,-\bsa^{(i)},\bsa^{(i)}}(N^{(i)})\right]
\end{align*}
and this gives the upper bound
\begin{align*}
\lefteqn{\left|\int_D f(\bsx) \overline\varphi(\bsx) \rd \bsx - \mathcal{A}_{\bsN,\bscK}(f) \right|} \nonumber \\
& \le \frac{1}{c} \|f\|_{p,s,\bsGamma} \sum_{i=1}^I \alpha_i c^{(i)} \frac{\delta+3 r^{(i)}-2}{N^{(i)}} \left[1 + 5 \cdot 2^t \log(2 N^{(i)})  G_{\bsgamma,q,-\bsa^{(i)},\bsa^{(i)}}(N^{(i)})\right].
\end{align*}
This leads to a bound on the last term of \eqref{eq1_thm4}, and thus the result follows. 
\end{proof}

\begin{remark}\rm
If we use unitary matrices $\bU_i$, then the required conditions on $f$ are stronger, because we require higher order derivatives. The weights in the norm of $f$ are now order-dependent weights. If the integrand $f$ does not have higher order derivatives or has some properties with respect to the weights which not only depend on $|\uu|$, then using the linear transformation may be a disadvantage. In such a case one may be better off setting $\bU_i = \bI$, the identity matrix.
\end{remark}

\section{Numerical experiments}\label{sec:numerical}

In the following, we perform numerical tests for two problems. In each instance, the underlying QMC point set is derived from a Sobol sequence as implemented in Matlab.

\subsection{Two-dimensional case}
We first consider a two-dimensional test case on the interval 
$[-5, 5]^2$, in which the unnormalized density $\pi$ is defined as
\begin{align*}
\pi(\bsx) = & \exp\left( - \|\bsx\|_2^2 -\frac2\sigma \left[ \left(\frac32 -\frac23 x_1 \right)^2 + 50\left(- \left(\frac23 x_1 - \frac12\right)^2 + x_2 - \frac12 \right)^2 \right.\right.  \\
& \hspace{3.2cm}\left. \left. + \left(\frac32 + \frac23 x_1 \right)^2 + 50\left(- \left( \frac23 x_1 + \frac12\right)^2 - x_2 - \frac12 \right)^2 \right] \right).
\end{align*}
As shown by the contours of $\pi$ in the top left plot of Fig.~\ref{fig:twod}, the target density concentrates on local regions and has nonlinear interaction between coordinates $x_1$ and $x_2$. We apply the adaptive hat function approximation (cf. Section~\ref{sec:3.2}) and the combined algorithm (cf. Section~\ref{sec:5.3}) to approximate $\pi$, and then test the convergence of the resulting QMC integration rules using Genz functions \cite{Genz84}. We consider the following instances of Genz functions
\begin{align*}
f_1(\bsx) = & \prod_{j=1}^s \left( \frac{1}{c_j^2} + \left(\frac{x_j+5}{10} + w_j\right)^2\right)^{-1}  ,\\
f_2(\bsx) = & \left(1 + \sum_{j=1}^s c_j \frac{x_j+5}{10}\right)^{-s-1},\\
f_3(\bsx) = & \exp\left( - \sum_{j=1}^s c_j \left|\frac{x_j+5}{10}-w_j\right| \right),
\end{align*}
for some constant vectors $\bsc, \bsw \in \RR^s$. Note that the original Genz functions are defined in the hypercube $[0,1]^s$ but here we have the interval $[-5, 5]^2$. So we rescale the above definitions to make them equivalent to the original ones. In this example, we set the constants to be $\bsc = (0.3, 0.6)$ and $\bsw = (0.25, 0.7)$.

\begin{figure}
\centering
\includegraphics[width=0.49\linewidth]{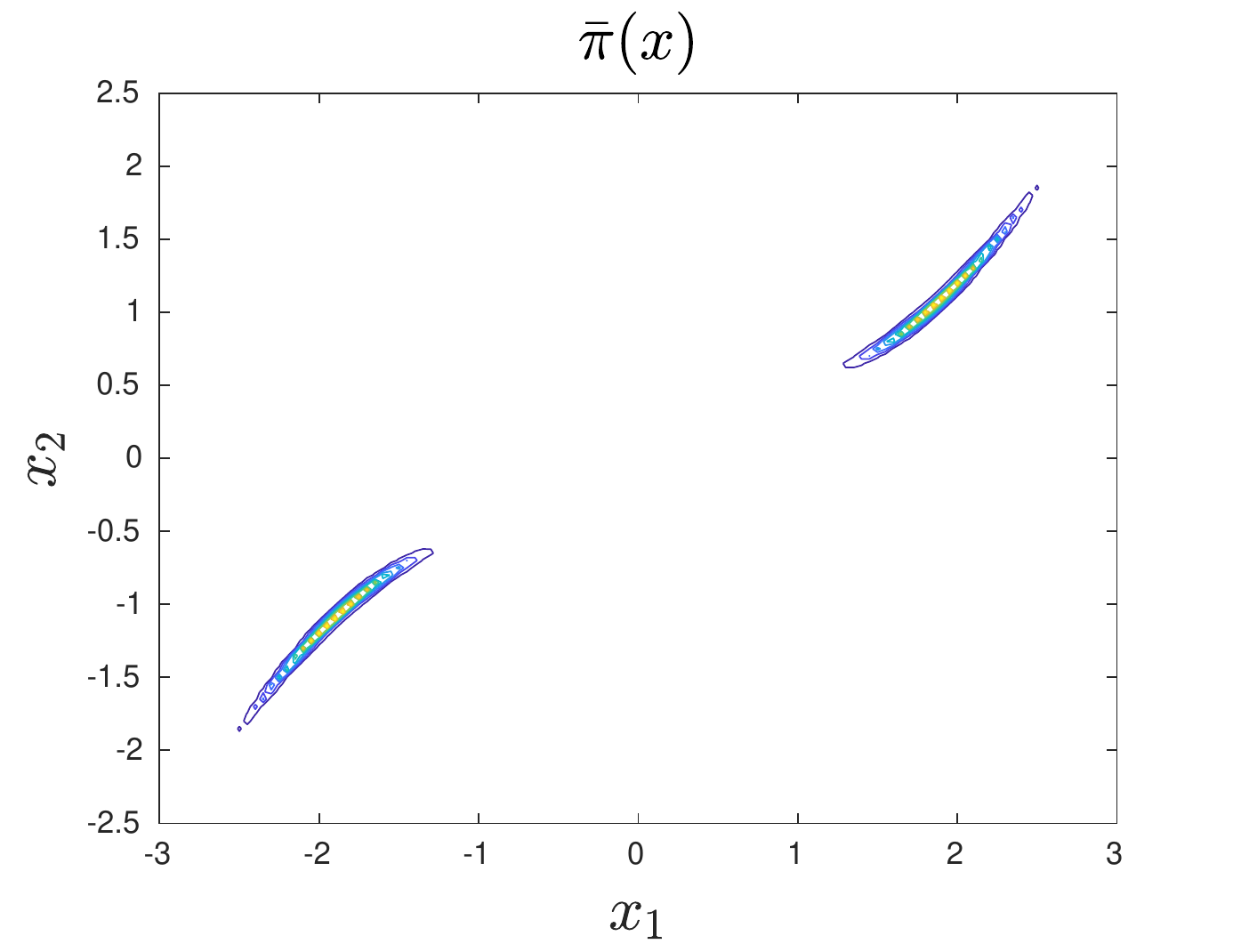}
\includegraphics[width=0.49\linewidth]{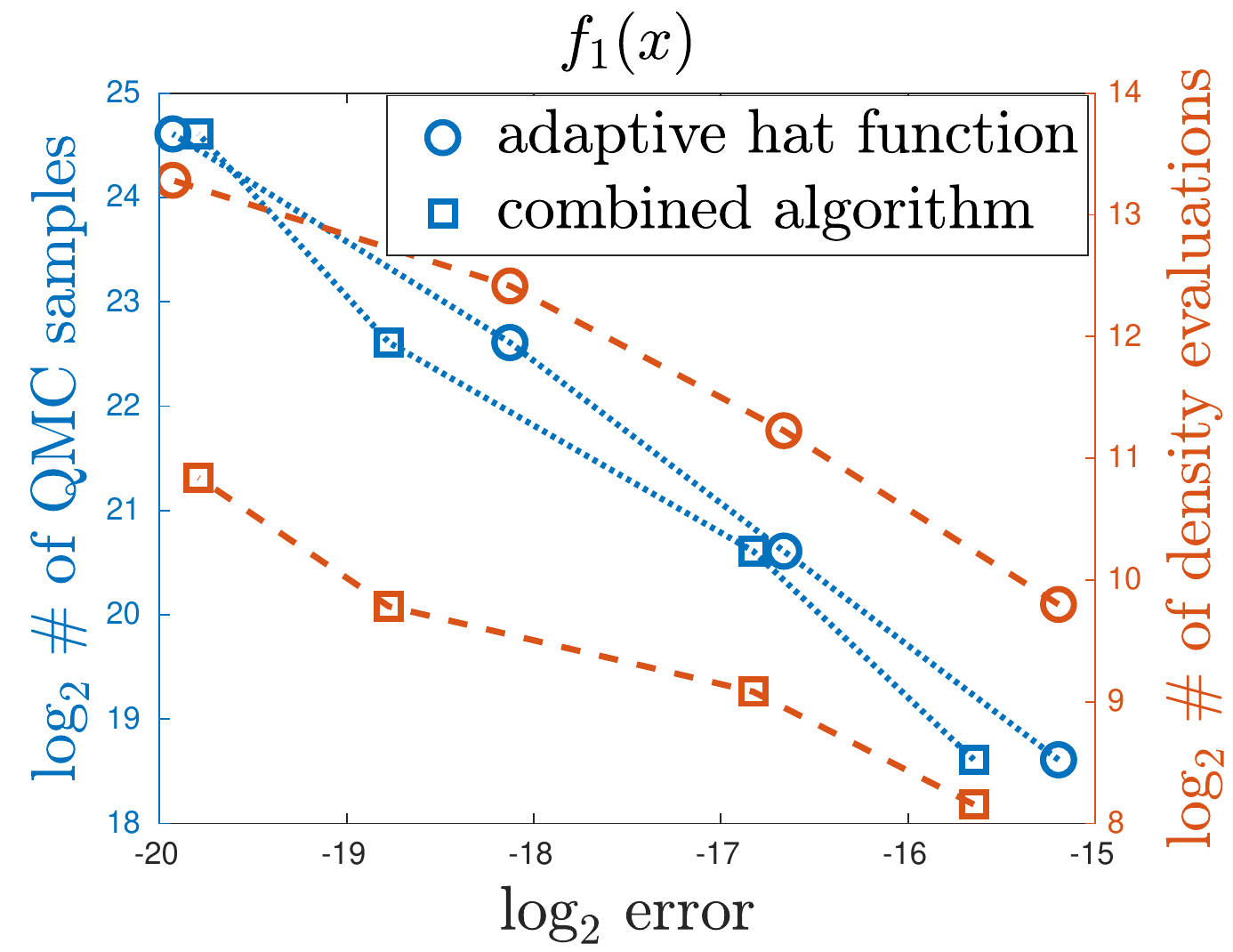}\\
\includegraphics[width=0.49\linewidth]{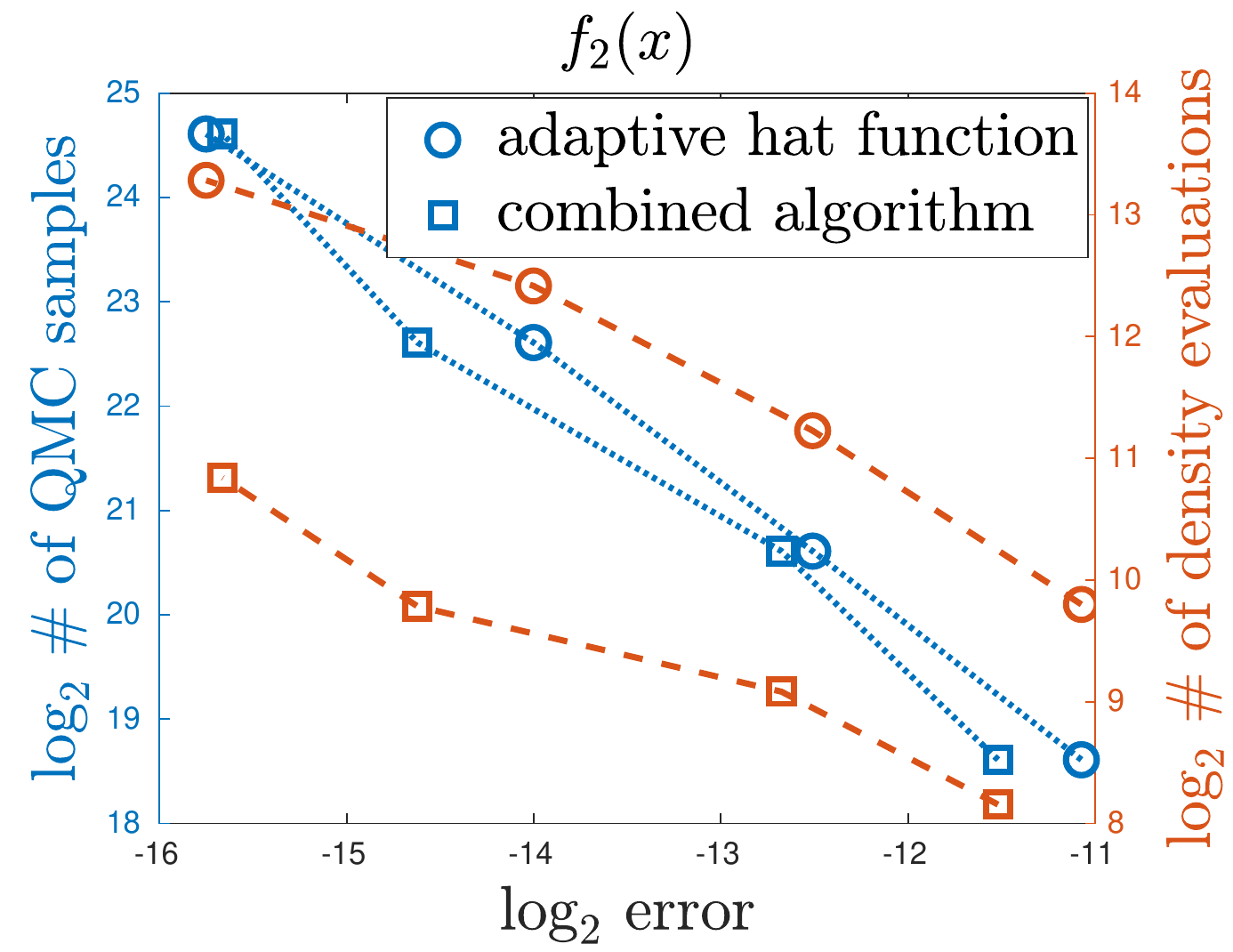}
\includegraphics[width=0.49\linewidth]{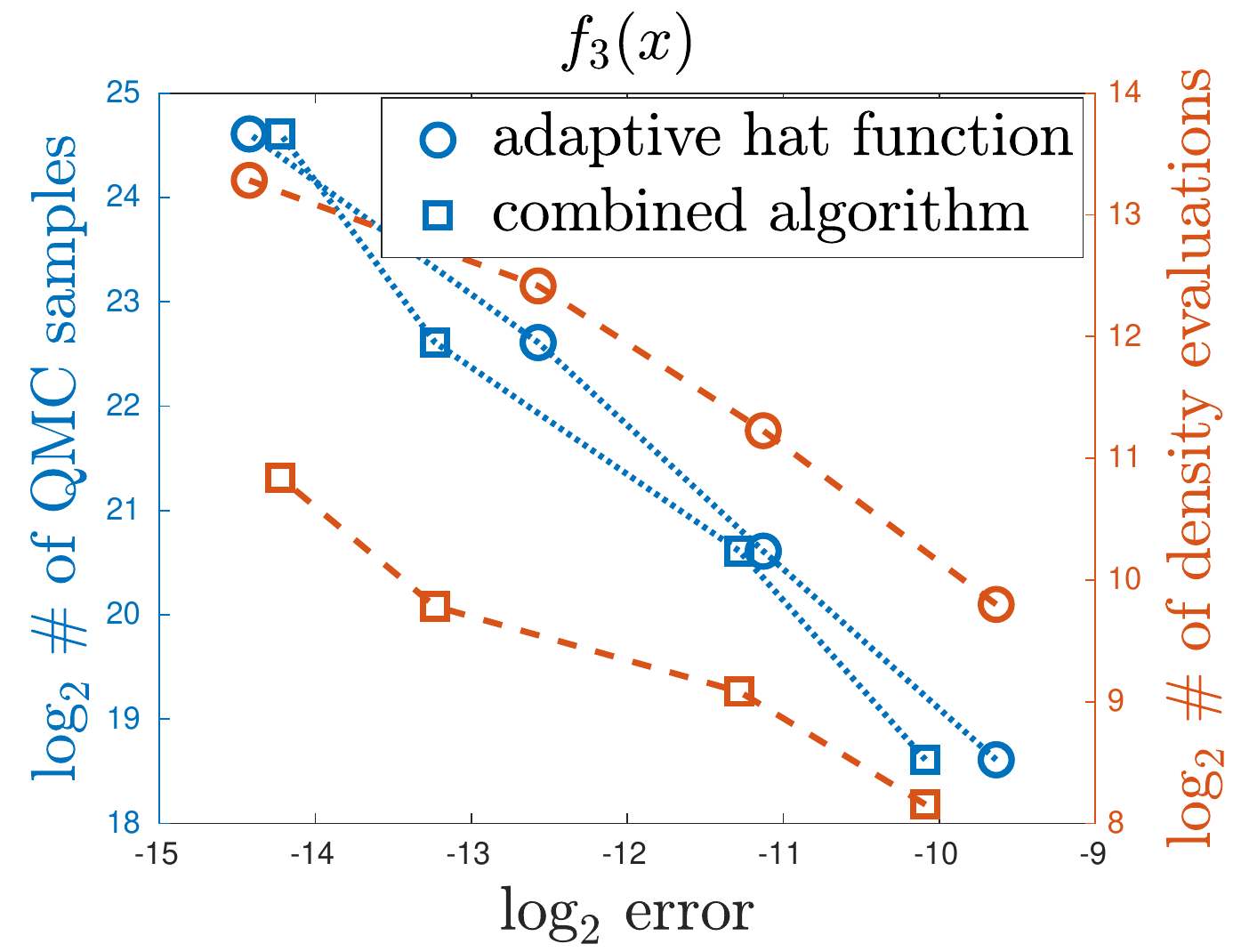}
\caption{The two-dimensional test case with three instances of Genz functions.}\label{fig:twod}
\end{figure}

We take a double refinement strategy to test the convergence of the QMC integration rules associated with the adaptive hat function approximation (cf. Section~\ref{sec:3.2}) and the combined algorithm (cf. Section~\ref{sec:5.3}). We first construct the hat function approximations with a sequence of decreasing target error threshold $\epsilon^{(k)} = 4^{-k} \cdot 5 \cdot 10^{-4}$ for $k\in\mathbb{N}_0$. Then, for each of the approximations in the sequence, indexed by $k\in\mathbb{N}_0$, we use the corresponding QMC integration rule with a total of $N^{(k)} = 4^{k+1} \cdot 10^5$ number of QMC samples. The estimated integration errors, the number of target density function evaluations used for building the hat function approximations, and the number of QMC samples are reported in Fig.~\ref{fig:twod}, in which the results obtained using the adaptive hat function approximation and the combined algorithm  are indicated by circles and squares, respectively. For all three instances of Genz functions, we observe a similar error decay rate of the QMC integration rules---it is about $N^{-0.8}$ for the adaptive hat function approximation and about $N^{-0.7}$ for the combined algorithm. While both methods have comparable convergence rates, we observe that the combined algorithm requires an order of magnitude less number of target density evaluations to build the approximation compared to the adaptive hat function approximation. 

\subsection{Bayesian estimation and prediction of a predator-prey system}

The predator-prey model is a system of coupled ordinary differential equations (ODEs) frequently used to describe the dynamics of biological systems. The populations of predator (denoted by $Q$) and prey (denoted by $P$) change over time according to a pair of ODEs
\begin{equation}\label{eq:pp}
\left\{\begin{array}{ll}
\displaystyle \frac{{\rm d} P}{{\rm d} t} & = \displaystyle \rho_P P \Big(1 - \frac{P}{K}\Big) - u \Big(\frac{PQ }{\alpha + P}\Big), \vspace{0.3em} \\
\displaystyle \frac{{\rm d} Q}{{\rm d} t} & = \displaystyle v \Big(\frac{PQ }{\alpha + P}\Big) - \rho_Q Q, 
\end{array}\right.
\end{equation}
with initial conditions $P(t=0) = P_0$ and $Q(t=0) = Q_0$.
The dynamical system is controlled by several parameters.
In the absence of the predator, the population of the prey evolves according to the logistic equation characterised by $\rho_P$ and $K$. In the absence of the prey, the population of the predator decreases exponentially at a rate of $\rho_Q$. In addition, the two populations have a nonlinear interaction characterised by $u$, $v$ and $\alpha$.
In this example, we assume that the initial condition $(P_0, Q_0)$, the parameter $u$, and the parameter $v$ are known. They take values $(P_0, Q_0, u, v) = (50, 5, 1.2, 0.5)$. The goal is to estimate the rest of the four parameters
\[
\bsx = ( \rho_P, K, \alpha, \rho_Q ),
\]
from observed populations of the predator and prey at time instances $t_i$ for $i \in \{1, \ldots, n_T\}$.

We use the Bayesian framework. As the starting point, we assign a 
uniform prior density to $\bsx$, which is uniformly distributed on the interval $[\bsa, \bsb]$ with 
$
\bsa = (0.36, 60, 15, 0.18)
$
and
$
\bsb = (0.96, 160, 40, 0.48).
$
Let $\bsy \in \mathbb{R}^{2n_T}$ denote the observed populations of the predator and prey.
We define a forward model $G: [\bsa, \bsb] \rightarrow \mathbb{R}^{2n_T}$ in the form of $G(\bsx) = [P(t_i, \bsx), Q(t_i, \bsx)]_{i=1}^{n_T}$ to represent the populations of the predator and prey computed at $\{t_i\}_{i =1}^{n_T}$ for given parameters $\bsx$.
Assuming independent and identically distributed (i.i.d.) normal noise in the observed data, one can define the unnormalized posterior density
\[
\pi(\bsx) = \exp\left(-\frac{1}{2\sigma^2}\|G(\bsx) - \bsy\|^2_2\right),\quad \bsx \in [\bsa, \bsb],
\]
where $\sigma$ is the standard deviation of the normally distributed noise. 
Synthetic observed data are used in this example. With $n_T = 13$ time instances $t_{i} = (i-1)\times 25/6$ and a given parameter $\bsx_{\rm true} = (0.6, 100, 25, 0.3)$, we generate synthetic noisy data $\bsy = \bsy_{\rm true} + \eta$, where $\eta$ is a realization of the i.i.d. zero mean normally distributed noise with standard deviation $\sigma = \sqrt{2}$.
To illustrate the behaviour of the posterior density, we plot the kernel density estimates of the marginal posterior densities in Fig.~\ref{fig:pp}.
Note that parameters are significantly correlated, making the posterior density function challenging to explore.

\begin{figure}
\centering
\includegraphics[width=0.7\linewidth]{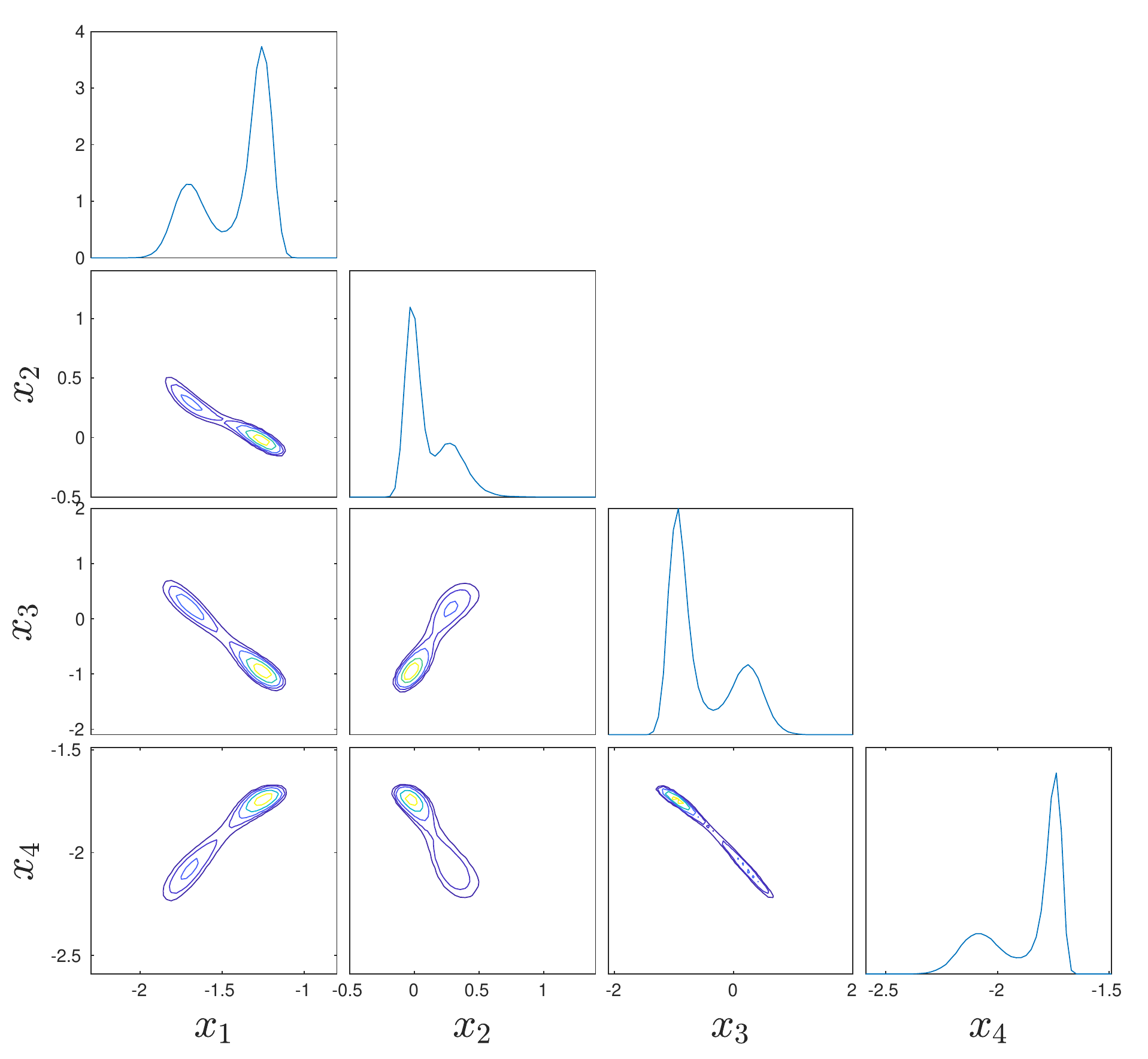}
\caption{The predator and prey example. Marginal posterior distributions.}\label{fig:pp}
\end{figure}

In this example, we separately estimate the risks of the prey population below a threshold $\tau_P=25$ and the predator population below a threshold $\tau_Q=15$, both at a given time $T=120$, which define functions of interest
\[
f_P(\bsx; \tau_P) = \mathbb{1}_{(-\infty, \tau_P]}(P(T, \bsx)) \quad  \text{and}  \quad f_Q(\bsx; \tau_Q) = \mathbb{1}_{(-\infty, \tau_Q]}(Q(T, \bsx)).
\]
Since the above risk functions $f_P$ and $f_Q$ do not have finite norm $\|f_P\|_{p,s,\bsgamma}, \|f_Q\|_{p,s,\bsgamma}$ (see \eqref{def:norm}), we also consider estimating the first, second, and third moments of $P(T, \bsx)$ and $Q(T, \bsx)$ at $T=120$. 

In this example, the adaptive hat function approximation (cf. Section~\ref{sec:3.2}) does not provide a valid approximation after using $10^6$ density evaluations. Thus, we only consider the combined algorithm (cf. Section~\ref{sec:5.3}) and the associated QMC integration rule. Similar to the previous example, we first construct the hat function approximations with a sequence of decreasing target error thresholds $\epsilon^{(k)} = 4^{-k} \cdot 5 \cdot 10^{-6}$ for $k\in\mathbb{N}_0$. Then, for each of the approximations in the sequence, indexed by $k\in\mathbb{N}_0$, we use the corresponding QMC integration rule with a total of $N^{(k)} = 4^{k} \cdot 10^5$ number of QMC samples. The number of target density function evaluations used for building the hat function approximations for each of the target error thresholds and marginal distributions of $P(T, \bsx)$ and $Q(T, \bsx)$ at $T=120$ are shown in Fig.~\ref{fig:pp_pred}. The estimated integration errors and the number of QMC samples are reported in Fig.~\ref{fig:pp_conv}, in which the results for predator and prey are indicated by circles and squares, respectively. For all the estimated moments, we observe a similar error decay rate of about $N^{-1}$, in which the rate of estimated moments of the predator is slightly higher than those of the prey. For the estimated risks, we observe asymptotically decreasing errors with increasing numbers of QMC points. However, the error decay rates are not smooth. Using the results given in Fig.~\ref{fig:pp_conv}, the estimated error decay rates for both risk functions $f_P$ and $f_Q$ are about $N^{-0.74}$.

\begin{figure}
\centering
\includegraphics[width=0.49\linewidth]{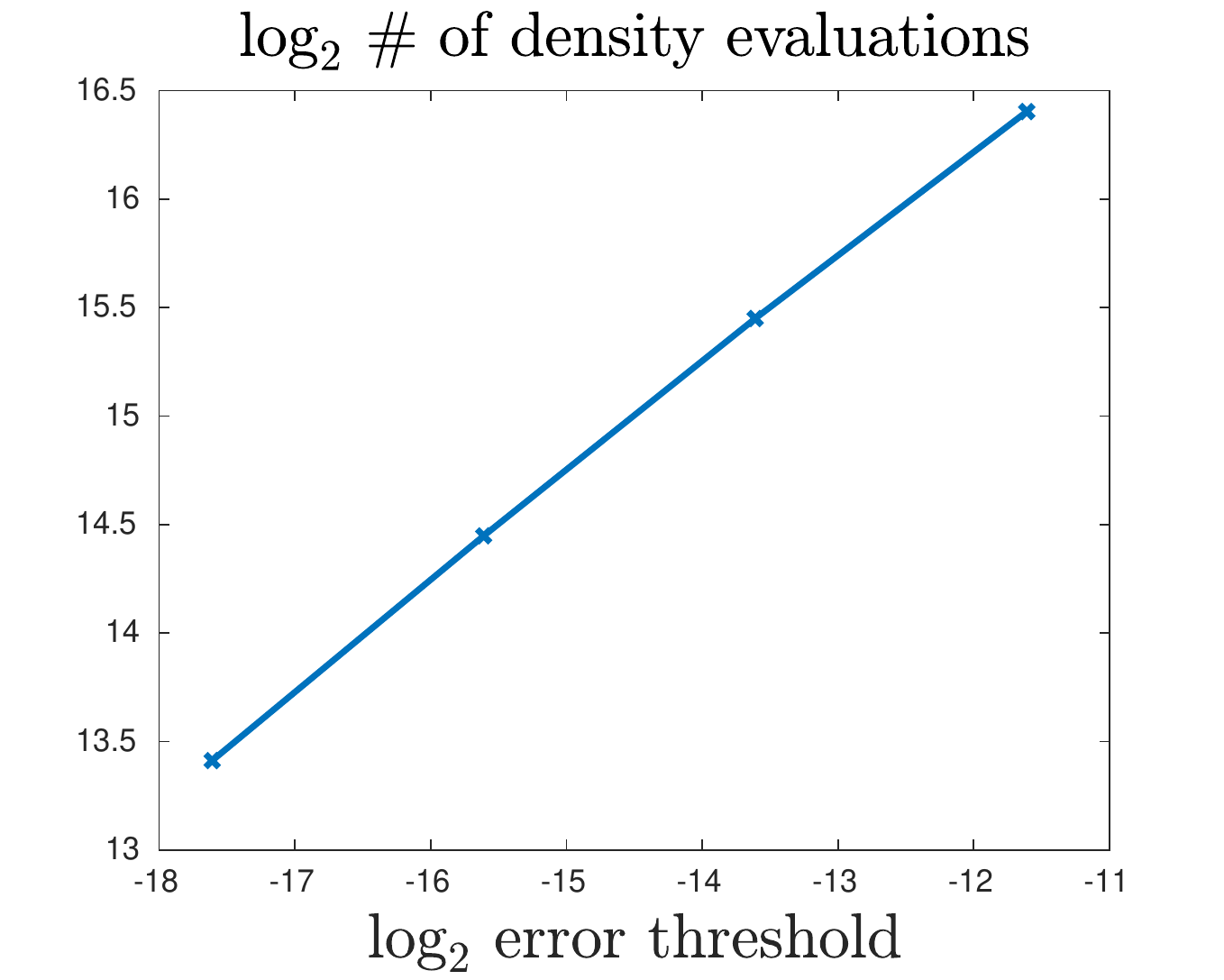}
\includegraphics[width=0.49\linewidth]{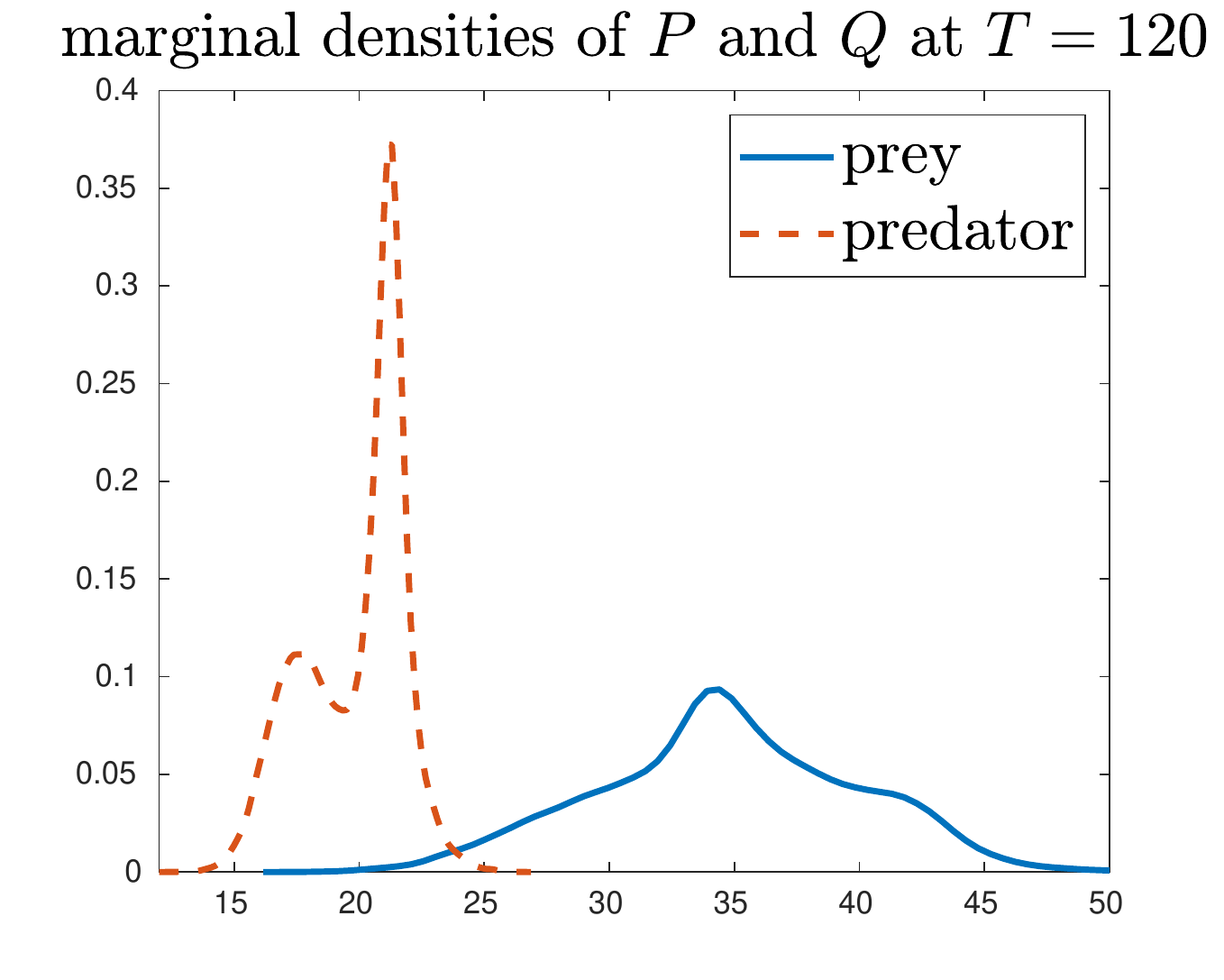}
\caption{The predator and prey example. Left: the number of target density function evaluations used for building the approximations for each of the target error thresholds. Right: marginal distributions of $P(T, \bsx)$ and $Q(T, \bsx)$ at $T=120$.}\label{fig:pp_pred}
\end{figure}

\begin{figure}
\centering
\includegraphics[width=0.49\linewidth]{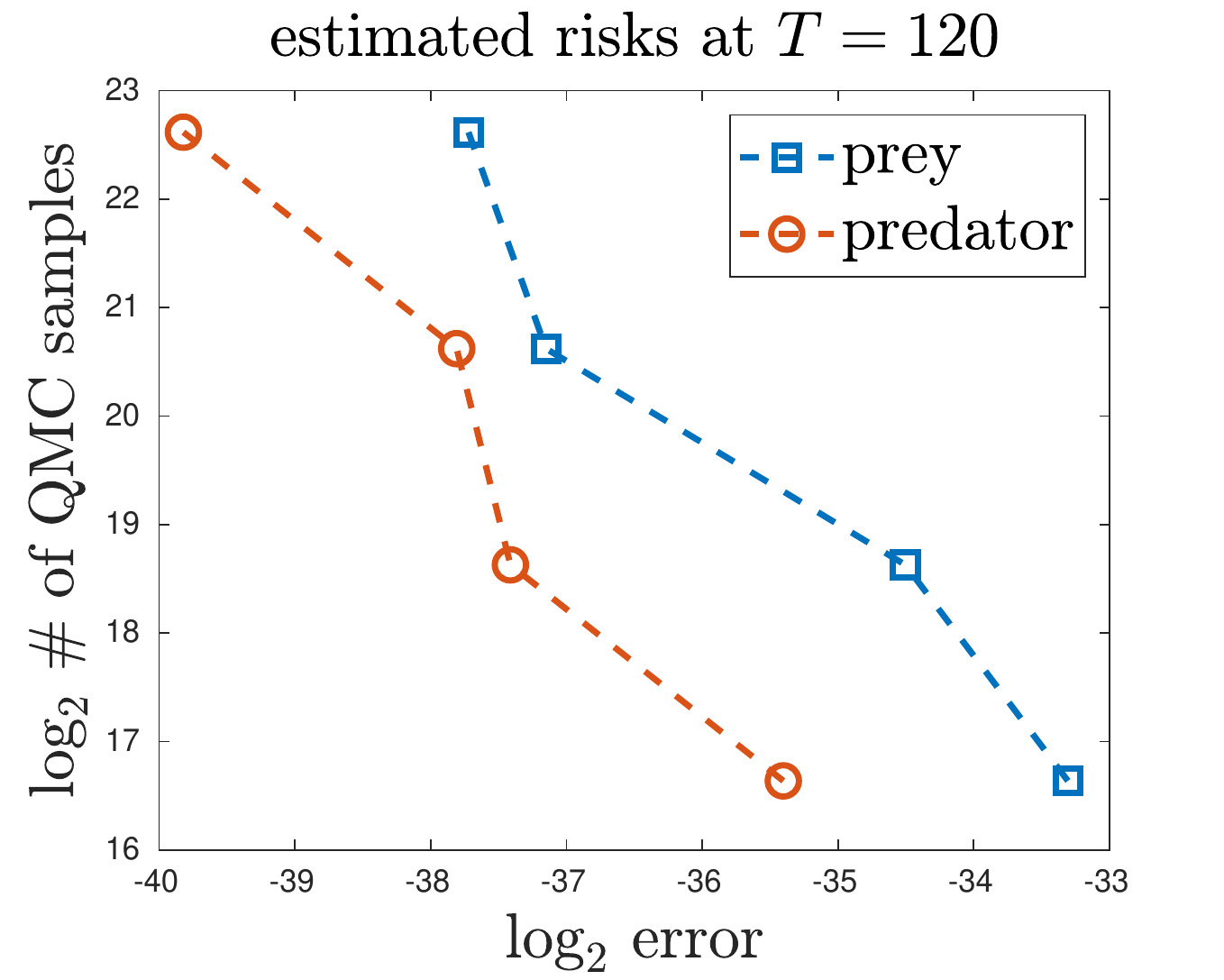}
\includegraphics[width=0.49\linewidth]{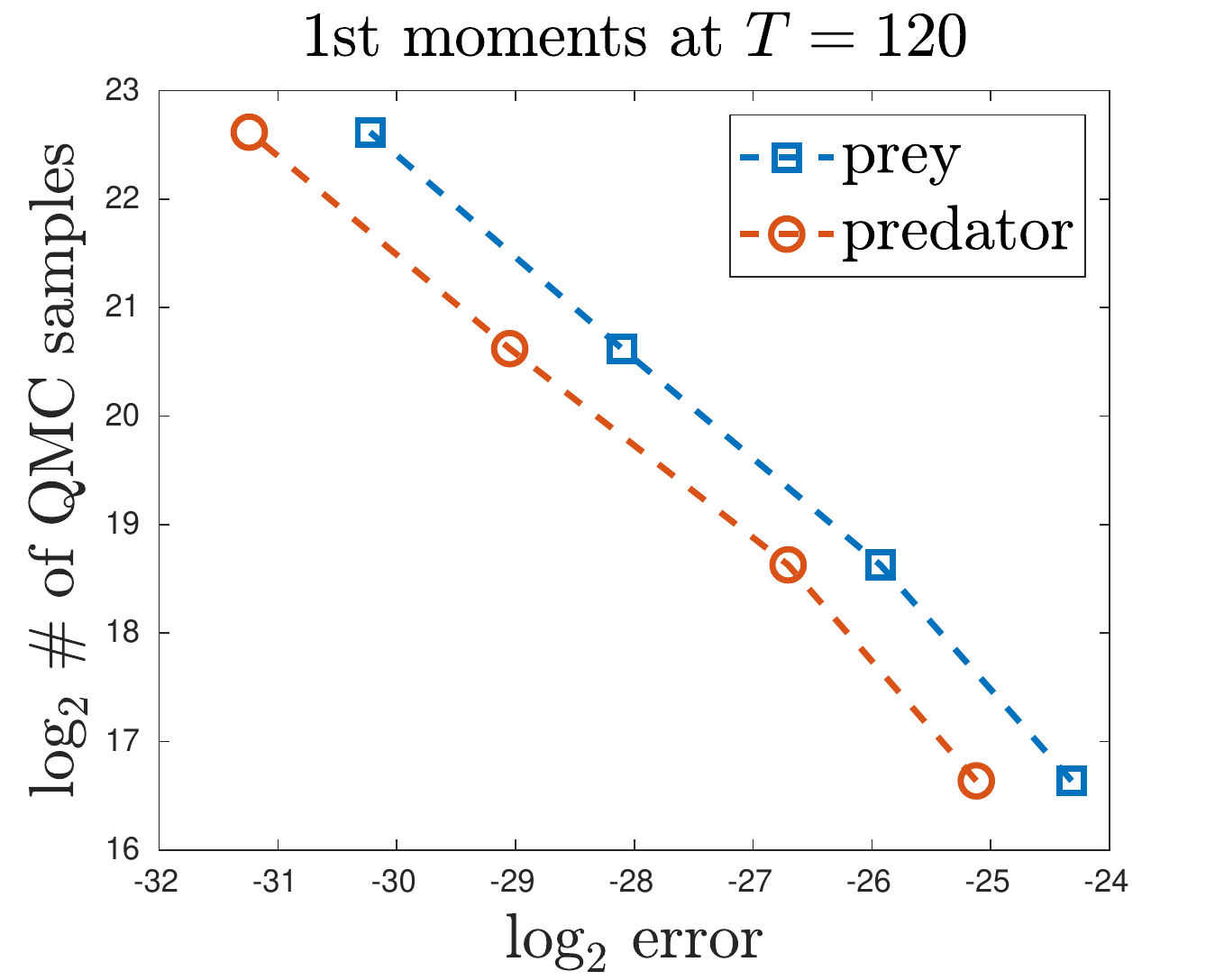}\\
\includegraphics[width=0.49\linewidth]{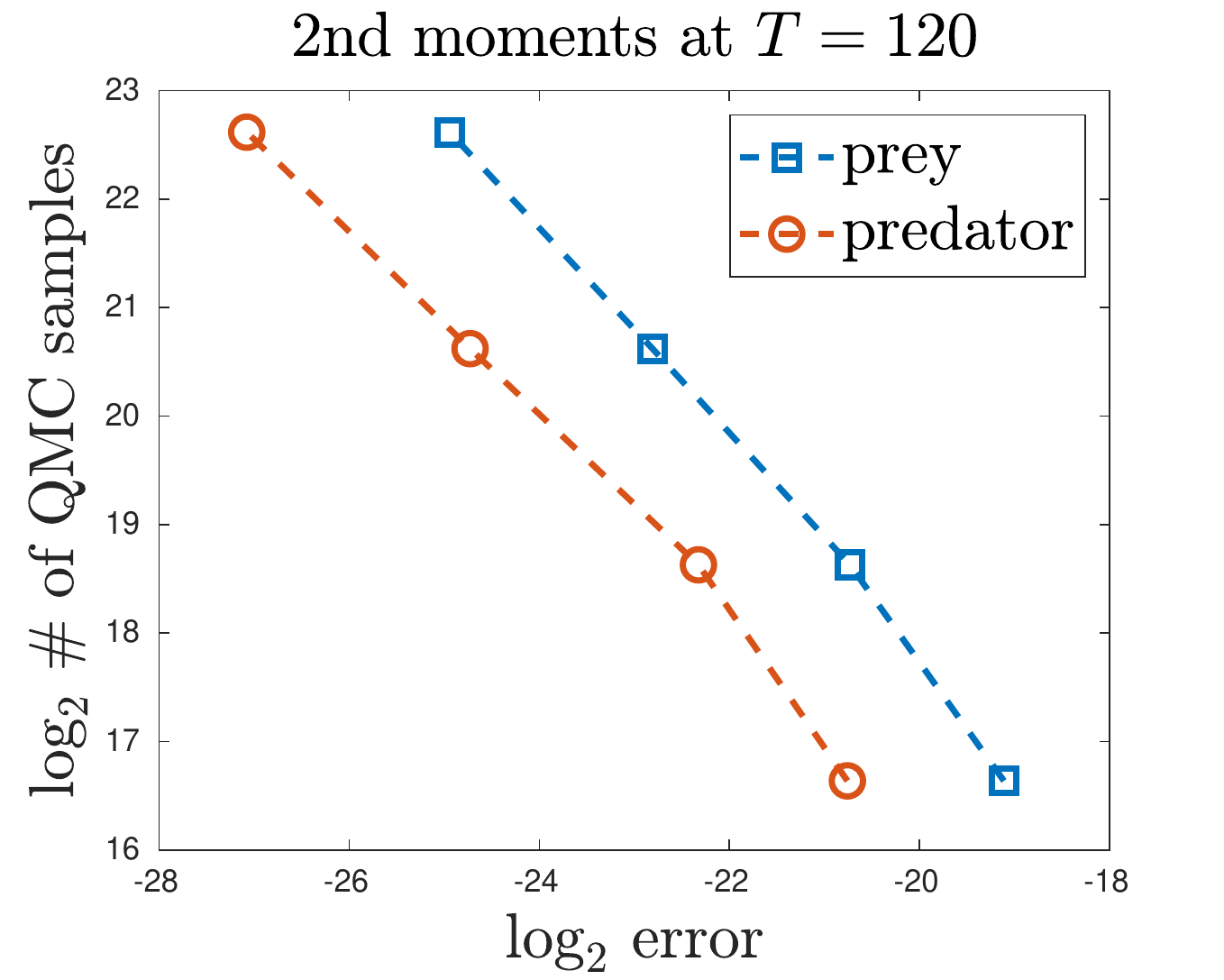}
\includegraphics[width=0.49\linewidth]{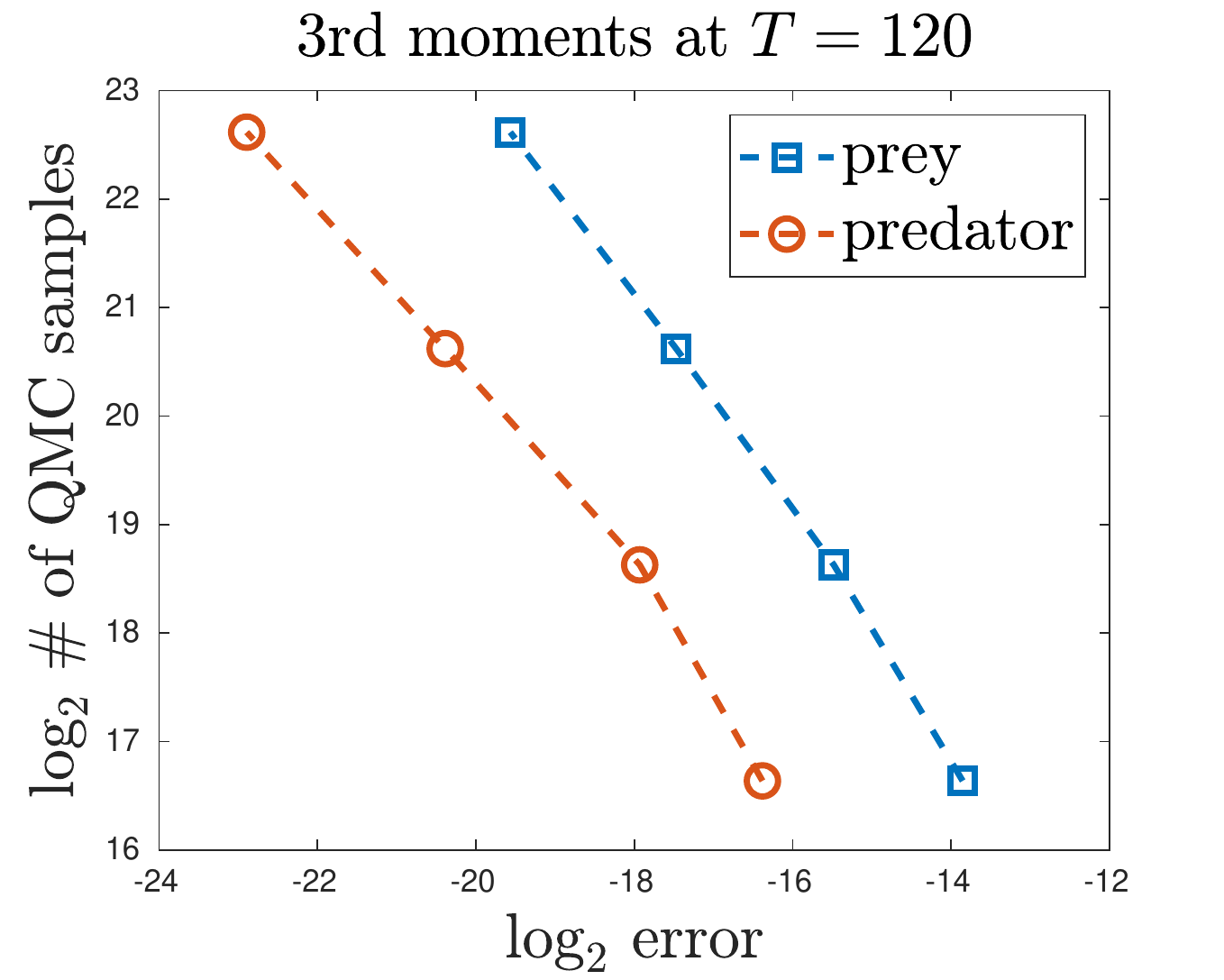}
\caption{The predator and prey example. Convergence of various estimators.}\label{fig:pp_conv}
\end{figure}

\begin{appendices}

\section{Auxiliary results}\label{sec:aux}

\textbf{The inverse cumulative distribution function for hat functions.} We consider $h_{\ell, m}$ as a function (non-normalized) probability density function. It can be checked that the normalization constant is $m/(b-a)$, hence we define
\begin{equation*}
\varphi_{\ell, m}(x) = \frac{m}{b-a} h_{\ell,m}(x),
\end{equation*}
and the cumulative distribution function
\begin{equation*}
\Phi_{\ell,m}(x) = \int_{-\infty}^x \varphi_{\ell, m}(z) \,\mathrm{d} z = \begin{cases} 0 & \mbox{if } x \le y_{\ell-1}, \\[0.5em] \frac{(x-y_{\ell-1})^2 }{2} \left(\frac{m}{b-a}\right)^2 & \mbox{if } y_{\ell-1} \le x \le y_{\ell}, \\[0.5em] 1 - \frac{(y_{\ell+1}-x)^2 }{2} \left(\frac{m}{b-a}\right)^2 & \mbox{if } y_{\ell} \le x \le y_{\ell+1} \\[0.5em] 1 & \mbox{if } x \ge y_{\ell+1}. \end{cases}
\end{equation*}
The inverse cumulative distribution function on the interval $[y_{\ell-1}, y_{\ell+1}]$ is given by
\begin{equation*}
\Phi^{-1}_{\ell, m}(z) = \begin{cases} y_{\ell-1} + \sqrt{2 z}\,\frac{ b-a}{m} & \mbox{if } 0 \le z \le 1/2, \\[0.5em] y_{\ell+1} - \sqrt{2 (1-z)}\, \frac{b-a}{m} & \mbox{if } 1/2 \le z \le 1. \end{cases}
\end{equation*}
Using $\Phi^{-1}_{\ell, m}$ we can transform a QMC point set for each $\ell$ and $m$.

\textbf{The inverse cumulative distribution function for adaptive hat functions.} Again we consider $h_{k_j, j}$ as a non-normalized probability density function. The normalization constant is
\begin{equation}\label{norm:const}
c_{k_j}^{(j)}=\begin{cases}
\frac{2}{y^{(j)}_{1}-a_j} & \mbox{if } k_j=0,\\[0.5em]
\frac{2}{y^{(j)}_{k_j+1}-y^{(j)}_{k_j-1}} & \mbox{if } 1 \le k_j < K_j,\\[0.5em]
\frac{2}{b_j-y^{(j)}_{K_j-1}} & \mbox{if } k_j = K_j. 
\end{cases}
\end{equation}
Define
\begin{equation}\label{eq_phikj}
\varphi_{k_j,j}(x) = c^{(j)}_{k_j} h_{k_j, j}(x),
\end{equation}
and the cumulative distribution functions
\begin{equation*}
\Phi_{0,j}(x) = \int_{-\infty}^x \varphi_{0,j}(z) \rd z = 
\begin{cases}
0 & \mbox{if } x \le a_j, \\[0.5em]
1 - \frac{(y^{(j)}_1 - x)^2}{(y^{(j)}_1 - a_j)^2} & \mbox{if } a_j \le x \le y^{(j)}_1, \\[0.5em]
1 & \mbox{if } x \ge y^{(j)}_1,
\end{cases}
\end{equation*}

\begin{equation*}
\Phi_{k_j,j}(x) = \int_{-\infty}^x \varphi_{k_j,j}(z) \rd z = 
\begin{cases}
0 & \mbox{if } x \le y^{(j)}_{k_j -1}, \\[0.5em]
\frac{(x - y^{(j)}_{k_j-1})^2}{(y^{(j)}_{k_j} - y^{(j)}_{k_j-1}) ( y^{(j)}_{k_j + 1} - y^{(j)}_{k_j-1} )} & \mbox{if } y^{(j)}_{k_j-1} \le x \le y^{(j)}_{k_j}, \\[0.5em]
1-\frac{(y_{k_j+1}^{(j)}-x)^2}{(y_{k_j+1}^{(j)}-y_{k_j}^{(j)})(y_{k_j+1}^{(j)}-y_{k_j-1}^{(j)})}& \mbox{if } y^{(j)}_{k_j} \le x \le y^{(j)}_{k_j+1}, \\[0.5em]
1 & \mbox{if } x \ge y^{(j)}_{k_j+1},
\end{cases}
\end{equation*}
for $k_j \in \{1,2,\ldots,K_j-1\}$, and
\begin{equation*}
\Phi_{K_j,j}(x) = \int_{-\infty}^x \varphi_{K_j,j}(z) \rd z = 
\begin{cases}
0 & \mbox{if } x \le y^{(j)}_{K_j-1}, \\[0.5em]
\frac{(x - y^{(j)}_{K_j-1})^2}{(b_j - y^{(j)}_{K_j-1})^2} & \mbox{if } y^{(j)}_{K_j-1} \le x \le b_j, \\[0.5em]
1 & \mbox{if } x \ge b_{j}.
\end{cases}
\end{equation*}
Thus the inverse cumulative distribution function on the interval $[y_{k_j-1}^{(j)}, y_{k_j+1}^{(j)}]$ for $k_j \in \{1,2,\ldots,K_j-1\}$ is given by
\begin{equation*}
\Phi^{-1}_{k_j,j}(z) = 
\begin{cases}
y_{k_j-1}^{(j)}+\sqrt{z \ (y_{k_j}^{(j)}-y_{k_j-1}^{(j)}) (y_{k_j+1}^{(j)}-y_{k_j-1}^{(j)})} & \mbox{if } 0 \le z \le \frac{y_{k_j}^{(j)}-y_{k_j -1}^{(j)}}{y_{k_j+1}^{(j)}-y_{k_j -1}^{(j)}}, \\[0.5em]
y_{k_j+1}^{(j)}-\sqrt{(1-z) \ (y_{k_j}^{(j)}-y_{k_j-1}^{(j)}) (y_{k_j+1}^{(j)}-y_{k_j-1}^{(j)})} & \mbox{if } \frac{y_{k_j}^{(j)}-y_{k_j -1}^{(j)}}{y_{k_j+1}^{(j)}-y_{k_j -1}^{(j)}} \le z \le 1.
\end{cases}
\end{equation*}
For the interval $[a_j, y^{(j)}_{1}]$ it is given by
\begin{equation*}
\Phi^{-1}_{0,j}(z) = y^{(j)}_1 - \sqrt{1-z} \ (y^{(j)}_1 - a_j)
\end{equation*}
and for the interval $[y^{(j)}_{K_j-1}, b_j]$ it is given by
\begin{equation*}
\Phi^{-1}_{K_j,j}(z) = y_{K_j-1}^{(j)} + \sqrt{z} \ (b_j - y^{(j)}_{K_j-1}).
\end{equation*}

\end{appendices}

\backmatter

\bmhead{Acknowledgments and declarations}

T. Cui is supported by the ARC Discovery Project DP210103092. J. Dick was supported by the ARC Discovery Project DP220101811 and the Special Research Program “Quasi-Monte Carlo Methods: Theory and Applications” funded by the Austrian Science Fund (FWF) Project F55-N26. F.~Pillichshammer is supported by the Austrian Science Fund FWF),Project F5509-N26, which is part of the Special Research Program ``Quasi-Monte Carlo Methods: Theory and Applications''. The authors declare that they have no conflict of interest.


\begin{thebibliography}{99}

\bibitem{ABD12} C.~Aistleitner, J.S.~Brauchart, and J.~Dick,  Point sets on the sphere $\mathbb{S}^2$ with small spherical cap discrepancy. Discrete Comput. Geom. 48 (2012), no. 4, 990--1024.

\bibitem{BO15} K.~Basu and A.B.~Owen, Low discrepancy constructions in the triangle. SIAM J. Numer. Anal. 53 (2015), no. 2, 743--761.

%


\bibitem{pmlr-v80-chen18f} W.Y.~Chen, L.~Mackey, J.~Gorham, F.-X.~Briol, and C.~Oates, Stein Points. In: Proc. 35th Int. Conf. Mach. Learn. 80 (2018), 844--853.

\bibitem{dempster1977maximum} A.P.~Dempster, N.M.~Laird, and D.B.~Rubin, Maximum likelihood from incomplete data via the EM algorithm. J. R. Stat. Soc. Series B Stat. 39 (1977), {1--22}.

\bibitem{DKP22} J.~Dick, P.~Kritzer, and F.~Pillichshammer, Lattice Rules -- Numerical Integration, Approximation, and Discrepancy. With an Appendix by Adrian Ebert. Springer Series in Computational Mathematics, 58. Springer, Cham, 2022.
 
\bibitem{DKS13} J.~Dick, F.Y.~Kuo, and I.H.~Sloan, High-dimensional integration: the quasi-Monte Carlo way. Acta Numer. 22 (2013), 133--288.

\bibitem{DP10} J.~Dick and F.~Pillichshammer, Digital Nets and Sequences -- Discrepancy Theory and Quasi-Monte Carlo Integration. Cambridge University Press, Cambridge, 2010.

\bibitem{DP2020} J.~Dick and F.~Pillichshammer, Weighted integration over a hyperrectangle based on digital nets and sequences. J. Comput. Appl. Math. 393 (2021), paper ref. 113509, 25 pp.

%

\bibitem{DAFS19} S.~Dolgov, K.~Anaya-Izquierdo, C.~Fox, and R.~Scheichl, Approximation and sampling of multivariate probability distributions in the tensor train decomposition. Statistics and Computing, 30 (2020), no. 3, 603--625.

%

\bibitem{Genz84} A.~Genz, Testing multidimensional integration routines, Proc. of International Conference on Tools, Methods and Languages for Scientific and Engineering Computation, 1984.

\bibitem{GC15} M.~Gerber and N.~Chopin, Convergence of sequential quasi-Monte Carlo smoothing algorithms. Bernoulli 23 (2017), no. 4B, 2951–2987.

\bibitem{GC17} M.~Gerber and N.~Chopin, Sequential quasi Monte Carlo. J. R. Stat. Soc. Ser. B. Stat. Methodol. 77 (2015), no. 3, 509–579.

\bibitem{GSY17} T.~Goda, K.~Suzuki, and T.~Yoshiki, Quasi-Monte Carlo integration for twice differentiable functions over a triangle. J. Math. Anal. Appl. 454 (2017), no. 1, 361--384.

\bibitem{GM15} J.~Gorham and L.~Mackey, Measuring sample quality with Stein's method. In: Advances in Neural Information Processing Systems, 28, 2015, pp. 226--234, Curran Associates, Inc.

\bibitem{HLD04} W.~Hoermann, J.~Leydold, and G.~Derflinger, Automatic Nonuniform Random Variate Generation. Springer Series in Statistics and Computing. Springer, Berlin, 2004, 

\bibitem{JMY90} M.~Johnson, L.~Moore, and D.~Ylvisaker, Minimax and maximin distance designs. Journal of Statistical Planning and Inference 26 (1990), 131--148.

\bibitem{KDSWW08} F.Y.~Kuo, W.T.M.~Dunsmuir, I.H.~Sloan, M.P.~Wand, and R.S.~Womersley, Quasi-Monte Carlo for highly structured generalised response models. Methodol. Comput. Appl. Probab. 10 (2008), no. 2, 239--275.

%

\bibitem{LMT10} P.~L'Ecuyer, D.~Munger, Ch.~L\'{e}cot, and B.~Tuffin, Sorting methods and convergence rates for Array-RQMC: some empirical comparisons. Math. Comput. Simulation 143 (2018), 191--201.

\bibitem{L09} Ch.~Lemieux, Monte Carlo and Quasi-Monte Carlo Sampling. Springer Series in Statistics. Springer, New York, 2009.

\bibitem{LP14} G.~Leobacher and F.~Pillichshammer, Introduction to Quasi-Monte Carlo Integration and Applications. Compact Textbooks in Mathematics. Birkh\"auser/Springer, Cham, 2014. 

%


\bibitem{Liu2008} J.S.~Liu, Monte Carlo Strategies in Scientific Computing. Springer Science \& Business Media, 2008.

\bibitem{MJ18} S.~Mak and V.R.~Joseph, Support points. Ann. Stat. 46 (2018), 6A, 2562--2592.

\bibitem{MMPS16} Y.~Marzouk, T.~Moselhy, M.~Parno, and A.~Spantini, Sampling via measure transport: An introduction. In: Handbook of Uncertainty Quantification. Vol. 1, 2, 3, pp.~785--825, Springer, Cham, 2017.

\bibitem{NK14} J.A.~Nichols and F.Y.~Kuo, Fast CBC construction of randomly shifted lattice rules achieving $\mathcal{O}(n^{-1+\delta})$ convergence for unbounded integrands over $\RR^s$ in weighted spaces with POD weights. J. Complexity 30 (2014), no. 4, 444--468.

\bibitem{nie87} H.~Niederreiter, Point sets and sequences with small discrepancy. Monatsh. Math. 104 (1987), no. 4, 273--337.

\bibitem{niesiam} H.~Niederreiter, Random Number Generation and Quasi-Monte Carlo Methods. No. 63 in CBMS-NSF Series in Applied Mathematics. SIAM, Philadelphia, 1992.

%

%

%

\bibitem{RG2004} C.P.~Robert and C. George, Monte Carlo Statistical Methods. Springer, 2004.

\bibitem{SHGCN16} C.~Schretter, Z.~He, M.~Gerber, N.~Chopin, and H.~Niederreiter, Van der Corput and golden ratio sequences along the Hilbert space-filling curve. In: Monte Carlo and Quasi-Monte Carlo Methods, pp.~531--544, Springer Proc. Math. Stat., 163, Springer, Cham, 2016.

\bibitem{sisson2007sequential} S.A.~Sisson, Y.~Fan,  and M.M.~Tanaka, Sequential Monte Carlo without likelihoods. Proc. Natl. Acad. Sci., 106 (2007), no. 6, 1760--1765.

\bibitem{wu1983convergence} C.F.J.~Wu, On the convergence properties of the EM algorithm. Ann. Stat. (1983), 95--103.

%

%

%

%

%

%

%

\end{thebibliography}
\end{document}